\documentclass[british,a4paper,11pt,reqno]{amsart}

\usepackage{babel}
\usepackage{needspace}
\usepackage{setspace}
\usepackage{aliascnt}
\usepackage{varioref}
\labelformat{equation}{\textup{(#1)}}
\labelformat{enumi}{\textup{(#1)}}

\usepackage{amssymb,mathrsfs,verbatim}
\usepackage{subfigure}

\usepackage[all,arc]{xy}

\usepackage{ifpdf}
\ifpdf
\IfFileExists{xypdf}{
\xyoption{pdf}
}{}
\usepackage{microtype}
\fi

\usepackage[pdfusetitle, unicode=true, bookmarks=true, pdfborder={0 0 0}, colorlinks=false, breaklinks=true, linktoc=all]{hyperref}

\hypersetup{breaklinks=false, colorlinks=true, linkcolor=blue, citecolor=blue, urlcolor=blue, linktoc=page} % Coloured links

\usepackage{tikz, scalefnt}

\title[Canonical bases for Fock spaces and tensor products]
{Canonical bases for Fock spaces \\and tensor products}

\author{Joseph~Chuang}
\address{Department of Mathematics, City University London, Northampton Square, London EC1V 0HB, United Kingdom. {\rm Email: j.chuang@city.ac.uk}}
\author{Kai~Meng~Tan}
\address{Department of Mathematics, National University of Singapore, Block S17, 10 Lower Kent Ridge Road, Singapore 119076.  {\rm Email: tankm@nus.edu.sg}}
\thanks{The second author is supported by Singapore Ministry of Education Academic Research Fund R-146-000-172-112.}
\newcounter{marginnotes}
\let\oldmarginpar\marginpar
\renewcommand\marginpar[1]{\addtocounter{marginnotes}{1}
\-\oldmarginpar[\raggedleft\footnotesize #1]%
{\raggedright\footnotesize #1}}
\newcommand{\countnotes}
{\ifnum\value{marginnotes}=0 {}
 \else
   \begin{center}
   \color{red}\bf\Huge
   \ifnum\value{marginnotes}=1 {There is 1 margin note in total.}
   \else {There are \arabic{marginnotes} margin notes in total.}
   \fi
    \end{center}
 \fi}
\usepackage{color}
\definecolor{light}{rgb}{0.6,0.6,0.6}
\usepackage[normalem]{ulem}
\usepackage{listings}
\lstnewenvironment{eml}
{\mbox{}\lstset{basicstyle=\footnotesize,breaklines=true,breakautoindent=false,breakindent=0pt,breakatwhitespace=true,frame=single,frameround=ffff,flexiblecolumns=true}}
{}
\theoremstyle{plain}
\newtheorem{theorem}{Theorem}[section]
\newcommand{\newautoreftheorem}[2]{
\newaliascnt{#1}{theorem}\newtheorem{#1}[#1]{#2}\aliascntresetthe{#1}%
\expandafter\def\csname #1autorefname\endcsname{#2}}

\newautoreftheorem{prop}{Proposition}
\newautoreftheorem{lemma}{Lemma}
\newautoreftheorem{cor}{Corollary}
\newautoreftheorem{conj}{Conjecture}
\newautoreftheorem{thm}{Theorem}
\newtheorem*{thm*}{Theorem}

\theoremstyle{definition}
\newautoreftheorem{definition}{Definition}
\newtheorem*{rem}{Remark}
\newautoreftheorem{example}{Example}
\newautoreftheorem{examples}{Examples}
\newautoreftheorem{notation}{Notation}
\newautoreftheorem{convention}{Convention}
\newtheorem*{notation*}{Notation}

\numberwithin{equation}{section} % Number equations by section
\numberwithin{figure}{section}   % Number figures by section

\def\hbar{\bar{h}}

\def\Ggl{{\mathfrak{gl}}}
\def\Gsl{{\mathfrak{sl}}}

\def\CF{{\mathcal{F}}}

\def\CP{{\mathcal{P}}}

\def\BB{{\mathbf{B}}}
\def\BC{{\mathbb{C}}}

\def\BN{{\mathbb{N}}}

\def\BS{{\mathbf{S}}}
\def\BT{{\mathbf{T}}}
\def\BU{{\mathbf{U}}}

\def\BX{{\mathbf{X}}}

\def\BZ{{\mathbb{Z}}}
\def\Ba{{\mathbf{a}}}
\def\Bb{{\mathbf{b}}}
\def\Bc{{\mathbf{c}}}
\def\Bd{{\mathbf{d}}}

\def\Bn{{\mathbf{n}}}
\def\Bs{{\mathbf{s}}}
\def\Bt{{\mathbf{t}}}

\def\Hom{\operatorname{Hom}\nolimits}

\def\tp{{\tilde{p}}}

\def\bw#1{{{\bigwedge}\!}^{#1}}

\def\U{\BU}
\def\Up{\BU(\Bn)}

\def\wt{\operatorname{wt}}

\def\tp{\mathbin{\diamond}}
\def\bigtpsymbol{{\Huge \raisebox{-3pt}{$\tp$}}}
\def\bigtp{\mathop{\text{\bigtpsymbol}}}

\def\s{\mathfrak{s}}
\def\t{\mathtt{t}}
\def\B{\mathfrak{B}}
\def\Par{\operatorname{Par}}
\def\P{\mathfrak{P}}
\def\pie{\varrho}
\def\E{\varepsilon_{q=1}}

\begin{document}

\maketitle

\begin{abstract} We relate the canonical basis of the Fock space representation of the quantum affine algebra
$U_q(\widehat\Ggl_{n})$, as defined by Leclerc and Thibon \cite{LT}, to the canonical basis of its restriction to
$U_q(\Gsl_{n})$, regarded as a based module in the sense of Lusztig. More generally we consider the restriction to any parabolic subalgebra. We deduce results on decomposition numbers and branching coefficients of Schur algebras over fields of positive characteristic, generalising those of Kleshchev \cite{Kl} and of Tan and Teo \cite{TT}.
\end{abstract}

\section{Introduction}
The complete determination of the decomposition numbers of the symmetric groups and Schur algebras in positive characteristic $p$ is a well-known and longstanding open problem, for which a complete solution does not seem to be forthcoming.  Related to these decomposition numbers are the $q$-decomposition numbers arising from the canonical basis for the Fock space representation of the quantum affine algebra $U_q(\widehat{\Ggl}_n)$.  These $q$-decomposition numbers, as conjectured by Leclerc and Thibon \cite{LT} and shown by Varagnolo and Vasserot \cite{VV}, are polynomials in $q$ with nonnegative integer coefficients and when evaluated at $q=1$ give the corresponding decomposition numbers for the $v$-Schur algebra in characteristic zero where $v$ is a primitive $n$-th root of unity.  As shown by James \cite{Jam}, when $n =p$, the decomposition matrix for the Schur algebra can be obtained by postmultiplying the decomposition matrix of the $v$-Schur algebra in characteristic zero by an adjustment matrix which has nonnegative integer entries and is unitriangular when the indexing set is suitably ordered.  As such, the $q$-decomposition numbers provide a first approximation to the decomposition numbers of the Schur algebras.

James's Conjecture asserts that this first approximation is in fact an equality whenever the indexing partitions have $p$-weight less than $p$. Even though the conjecture is now known to be false in general \cite{W}, it has been proved in some cases, such as in Rouquier blocks \cite{CT} and in blocks with $p$-weight less than 5 \cite{F1,F2}.  %On the other hand, Fayers \cite{F} shows that for each $w$ which is  at least $p$, the first approximation is not an equality for some indexing partitions having $p$-weight $w$.
 It has also been shown that the first approximation is in fact an equality in many cases irrespective of the $p$-weight of the indexing partitions; see for example \cite{beyond,TT}.
 %To date, there is no conjecture on when the first approximation is not an equality.

In \cite{Kl}, Kleshchev introduces the combinatorics of sign sequences, and uses it to describe the decomposition number $d_{\lambda\mu}$ when the partition $\lambda$ is obtained from $\mu$ by moving one node.  Subsequently, the present authors and Miyachi \cite{CMT2} provide closed formulas for the corresponding $q$-decomposition number $d_{\lambda \mu}(q)$ using the same combinatorics.  More recently, the second author and Teo \cite{TT} provide closed formulas for $d_{\lambda\mu}(q)$ when $\lambda$ is obtained from $\mu$ by moving any number of nodes as long as all of them have the same $n$-residue, and show that, when $n = p$, $d_{\lambda\mu}(1) = d_{\lambda\mu}$.  The astute reader of \cite{TT} who is familiar with the work of Frenkel and Khovanov in \cite{FK} will be struck by the uncanny similarity between the last closed formulas and those describing the canonical bases of tensor powers $V_2^{\otimes d}$ of the natural two-dimensional representation $V_2$ of the quantum enveloping algebra $U_q(\Gsl_2)$, although the former is formulated using the combinatorics of sign sequences while the latter is described by graphical calculus.  It is natural to attempt to find out the exact relationship between the latter canonical bases with that of the Fock space representation.  This is the main motivation of our work appearing in this paper.

We briefly describe our results here. For our purposes it suffices to consider the subalgebra
 $\U = U_q(\widehat{\Gsl}_n)$ of $U_q(\widehat{\Ggl}_n)$.
Let $\Bn = (n_1,\dotsc,n_r)$ be a tuple of positive integers such that $n_1 + \dotsb + n_r = n$, and let $\Up$ denote the subalgebra of $\U$ isomorphic to $U_q(\Gsl_{n_1}) \otimes \dotsb \otimes U_q(\Gsl_{n_r})$.  We show that the Fock space representation $\CF_s$ indexed by an integer $s$ (see \autoref{prelimFock} for formal definition), when restricted from $\U$ to $\Up$, has a natural decomposition, which corresponds to partitioning the standard basis of $\CF_s$ into subsets in a certain way.  Each summand $\CF_{\Bt}$ of the restriction is isomorphic to a tensor product of %exterior powers of the natural $n_j$-dimensional representation of $U_q(\Gsl_{n_j})$ and %
irreducible $\Up$-modules and hence, as a based
module in the sense of Lusztig \cite{Lu}, is equipped with a canonical basis.  We then show that a subset of the canonical basis of $\CF_{s}$ maps to this basis of $\CF_{\Bt}$ under the natural projection map $\CF_s \to \CF_{\Bt}$ (\autoref{T:decomp}).
The difficulty in directly relating the canonical bases
of $\CF_s$ and $\CF_{\Bt}$ stems from the differing nature of the corresponding `bar involutions' that fix basis vectors: in the Fock space, the involution is given by reversing $q$-wedges \cite{LT}, whereas for tensor products of based modules, it is defined through quasi-$R$-matrices \cite{Lu}.

 Subsequently, the results in \cite{FK} can be exploited to provide more information when $n_j = 2$ for some $j$.  In particular, this establishes the exact relationship between the canonical basis of the Fock space and that of $V_2^{\otimes d}$.  We also obtain closed formulas for some of the branching coefficients for $\CF_s$.

We then turn our attention to the Schur algebras.  Assuming the results of Kleshchev \cite{Kl} describing the branching coefficients $[{\rm Res} (L(\mu)) : L(\lambda)]$ when $\lambda$ is obtained from $\mu$ by removing a normal node, we obtain closed formulas for the decomposition numbers $d_{\lambda\mu}$ when $\lambda$ is obtained from $\mu$ by moving some nodes whose $p$-residues are pairwise non-adjacent (\autoref{C:Schur}), generalising the results of \cite{TT} on the decomposition numbers.  We also show that $[{\rm Res} (L(\mu)) : L(\lambda)]= 0$ whenever $\lambda$ is obtained from $\mu$ by removing a node and moving other nodes while preserving their $p$-residues, such that the $p$-residues of all these nodes are pairwise non-adjacent (\autoref{branching}).

We now indicate the layout of this paper.  We begin in the next section with a short account of the background theory which we require.  In \autoref{setup}, we prove some preliminary results which we shall require in a general setting to deal simultaneously with the quantized Fock spaces and the Grothendieck group of the Schur algebras.  In \autoref{RestrictToU(n)}, we describe the restriction of the quantized Fock space to $\Up$ as a direct sum of factors $\CF_{\Bt}$ isomorphic to tensor products of exterior powers of the natural $n_j$-dimensional representations of the $U_q(\Gsl_{n_j})$. In \autoref{compare}, we relate the canonical basis of the quantized Fock space to that of $\CF_{\Bt}$.  In \autoref{S:decompnumbers}, we apply an argument analogous to that used in \autoref{compare} to the Schur algebras and obtain closed formulas for the decomposition numbers mentioned above.

\section{Preliminaries} \label{prelim}

Denote $\BN = \{1,2,\dotsc,\}$ and $\BN_0 = \{ 0,1,\dotsc \}$.

\subsection{Partitions, $\beta$-numbers, abaci}

A partition $\lambda = (\lambda_1,\lambda_2,\dotsc)$ is an infinite weakly decreasing sequence of non-negative integers such that $\lambda_k = 0$ for all large enough $k$.  We write $|\lambda|$ for $\sum_{i=1}^{\infty} \lambda_i$, and denote the set of all partitions by $\CP$.

For $\lambda \in \CP$, define its Young diagram $[\lambda] = \{ (i,j) \in \BN^2 \mid j\leq \lambda_i \}$. The elements of $\BN^2$ are usually called nodes in this context.  A node $(a,b)$ is to the left of a node $(c,d)$ if $b<d$, in which case $(c,d)$ is to the right of $(a,b)$.  If $n \in \BN$, the $n$-residue of a node $(i,j)$ is the residue class of $j-i$ modulo $n$.  For convenience, for $k \in \mathbb{Z}$, we will say a node has $n$-residue $k$ if the $n$-residue of the node is the residue class of $k$ modulo $n$.

A node $\mathfrak{n} \in [\lambda]$ is removable if $[\lambda] \setminus \{\mathfrak{n}\} = [\mu]$ for some $\mu \in \CP$, in which case we also say that $\mathfrak{n}$ is an addable node of $[\mu]$ (or simply $\mu$).  The removable node is normal if it, as well as any of the removable or addable nodes with the same $n$-residue as and to the right of $\mathfrak{n}$, has at least as many removable nodes as addable nodes of the same $n$-residue to its right.

For $(a,b) \in [\lambda]$, define
$$
\mathfrak{h}_{a,b}(\lambda) = \{ (i,j) \in [\lambda] \mid i \geq a,\ j \geq \max(b,\lambda_{i+1}) \}.
$$
This is a rim hook of $[\lambda]$ (or simply $\lambda$).  Note that $[\lambda] \setminus \mathfrak{h}_{a,b}(\lambda) = [\nu]$ for some $\nu \in \CP$; we say that $\nu$ is obtained from $\lambda$ by unwrapping $\mathfrak{h}_{a,b}(\lambda)$ and $\lambda$ is obtained from $\nu$ by wrapping $\mathfrak{h}_{a,b}(\lambda)$.

A subset $B$ of $\BZ$ is a set of $\beta$-numbers if $|\BN_0 \cap B| + |\BZ_{<0} \setminus B|$ is finite.  Denote the collection of all sets of $\beta$-numbers by $\B$.  If $B \in \B$, then since it is bounded above, we can arrange its elements in decreasing order and obtain its associated $\beta$-sequence $\overline{B} = (\overline{B}_1,\overline{B}_2,\dotsc)$.
Write $\s(B) = |\BN_0 \cap B| - |\BZ_{<0} \setminus B|$, and for each $s \in \BZ$, let $\B_s = \{ B \in \B \mid \s(B) = s\}$.

%To each $B \in \B$, we associate the strictly decreasing sequence $\overline{B}$ of integers in $B$.

For each $\lambda \in \CP$ and $s \in \BZ$, define $\beta_s(\lambda) = \{ \lambda_i + s-i \mid i \in \BN \}$.  Then $\beta_s(\lambda) \in \B_s$.  In fact, $\lambda \leftrightarrow \beta_s(\lambda)$ gives a one-to-one correspondence between $\CP$ and $\B_s$.  More generally, we have a bijection between $\CP \times \BZ$ and $\B$ given by $(\lambda,s) \leftrightarrow \beta_s(\lambda)$.  For each $B \in \B$, write $\Par(B)$ for the partition such that $(\Par(B),\s(B)) \leftrightarrow B$.

The $n$-abacus was introduced by James to facilitate manipulations with rim hooks of size $n$ (see, e.g. \cite{JK}).
It has $n$ vertical runners, labelled $0$, $1$, $\dotsc$, $n-1$ from left to right, and infinitely many rows, labelled by $\BZ$ in an ascending order from top down.  The position on row $i$ and runner $j$ of the $n$-abacus is labelled $in+j$.  Thus the positions on the $n$-abacus are labelled by integers running left to right and top to bottom, with position $0$ in the leftmost runner.  We may display any subset $S$ of $\BZ$ on the $n$-abacus by placing a bead on position $x$ for each $x \in S$.  We do so especially for the elements $B \in \B$. See Figure~\ref{f:exampleabacus} for an example.

\begin{figure}[h]%\label{abacusfigure}
	\begin{tikzpicture}[y=-1cm, scale=0.8, every node/.style={font=\scalefont{1}}]
	
	%runner labels and runners
	\foreach \x in {0,...,8}
	{\node [above] at (\x, -3) {$\x$};
		\draw (\x,-2.5) --(\x,3.5);};
	
	%row labels
	\foreach \y in {-2,...,3}
	{\node [left] at (-1, \y) {$\y$};};
	
	%beads
	\foreach \x in
	{(0,-2), (1,-2), (2,-2), (3,-2), (4,-2), (5,-2), (6,-2), (7,-2), (8,-2),
		(0,-1), (1,-1), (2,-1), (3,-1), (4,-1), (5,-1), (6,-1), (7,-1), (8,-1),
		(0, 0),         (2, 0), (3, 0),         (5, 0),         (7, 0),
		(1, 1), (2, 1),         (4, 1),                 (7, 1), (8, 1),
		(0, 2),                 (3, 2),                 (6, 2),         (8, 2)
	}
	{\draw [fill] \x circle [radius=0.2];};
	
	\end{tikzpicture}
	\caption{Displaying the partition $\lambda=(13,12,10,8,8,8,6,5,5,3,2,1,1)$ on a $9$-abacus, with $s=14$.}
	\label{f:exampleabacus}
\end{figure}
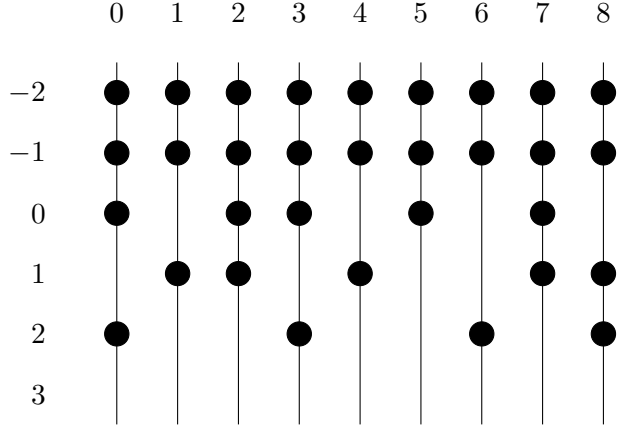

%We display $B$ on such an abacus by placing a bead on position $i$ for each $i \in B$.  %We identify $\lambda$ with $\beta_s(\lambda)$ and its display on the $n$-abacus. See \autoref{f:exampleabacus} for an example.\marginpar{In the figure choose appropriate partition and $s$, and then use same one for the figure in the section on restriction}

Moving a bead from position $a$ on the $n$-abacus display of $B$ to a vacant position $b$ produces the display of $B' = B \setminus \{a\} \cup \{b\}$, and $\s(B')=\s(B)$.  If $a>b$, then $\Par(B')$ can be obtained from $\Par(B)$ by unwrapping a rim hook of size $a-b$, namely $\mathfrak{h}_{x,y}(\Par(B))$ where $a = \overline{B}_x$, $\overline{B}_{z+1} < b < \overline{B}_{z}$, and $y = z + b+1-\s(B)$.
Conversely, if $\mu \in \CP$ is obtained from $\lambda$ by unwrapping a rim hook of size $c$, then $\beta_s(\mu) = \beta_s(\lambda) \setminus\{x\} \cup \{y\}$ for some $x \in \beta_s(\lambda)$ and $y \notin \beta_s(\lambda)$ such that $x-y = c$.

In particular, moving a bead on runner $i$ in the $n$-abacus display of $B$ to its vacant preceding position corresponds to removing a removable node of $n$-residue $i_s$, where $i_s$ is the residue class of $i-s$ modulo $n$, from $\Par(B)$.

The $n$-core of $\lambda$ is defined to be the partition obtained from $\lambda$ by successively removing rim hooks of size $n$. This is the partition obtained by sliding the beads in the $n$-abacus display of $\lambda$ up their respective runners as high up as possible, and is thus well-defined (i.e.\ independent of the order in which the rim hooks of size $n$ are removed).

Let $\geq$ denote the usual lexicographic ordering on the set of infinite sequences of integers, so that $(a_1,a_2,\dotsc) > (b_1,b_2,\dotsc)$ if and only if there exists $r \in \BN$ such that $a_k = b_k$ for all $k < r$ and $a_r > b_r$.  Then $\geq$ restricts to a total order on $\CP$ and on $\{ \overline{B} \mid B \in \B \}$ and hence on $\B$.
It is easy to see that, for $\lambda,\mu \in \CP$, the following statements are equivalent:
\begin{enumerate}
\item $\lambda \geq \mu$.
\item $\beta_s(\lambda) \geq \beta_s(\mu)$ for all $s \in \BZ$.
\item $\beta_s(\lambda) \geq \beta_s(\mu)$ for some $s \in \BZ$.
\end{enumerate}
Note that, from our definition, it is possible for $\lambda \geq \mu$ when $|\lambda| \ne |\mu|$, and for $B \geq B'$ when $\s(B) \ne \s(B')$.

We now define a coarser order on $\B$.  Let $B,B' \in \B$.  We write $B \to_n B'$ if $B' = B \setminus \{a,b\} \cup \{a-in, b+in\}$ for some $a,b\in B$ and $i \in \BN$ such that $a > b+in$ and $a-in,b+in \notin B$.  The Jantzen order $\geq_{J_n}$ on $\B$ is defined
as follows:  $B \geq_{J_n} B'$ if and only if there exist $B_0,\dotsc, B_t \in \B$ ($t \in \mathbb{N}_0$) such that $B_0 = B$, $B_t = B'$, and $B_{i-1} \to_n B_{i}$ for all $i = 1,\dotsc, t$.  We note that, if $\lambda, \mu \in \CP$, then $\beta_s(\lambda) \to_n \beta_s(\mu)$ for some $s \in \BZ$ if and only if $\beta_s(\lambda) \to_n \beta_s(\mu)$ for all $s \in \BZ$.  As such, we may define the relations $\to_n$ and $\geq_{J_n}$ on $\CP$:  if $\lambda, \mu \in \CP$, then $\lambda \to_n \mu$ (resp.\ $\lambda \geq_{J_n} \mu$) if and only if $\beta_s(\lambda) \to_n \beta_s(\mu)$ (resp.\ $\beta_s(\lambda) \geq_{J_n} \beta_s(\mu)$) for some $s \in \BZ$.  Note that $\lambda \to_n \mu$ if and only if $\lambda \geq \mu$ and $\mu$ is obtained from $\lambda$ by unwrapping a hook of size $in$ from $\lambda$ and then wrapping a hook of size $in$.

\subsection{Quantized enveloping algebras} \label{qea}

Let $\U=U_q(\widehat\Gsl_n)$ be the quantized enveloping algebra
of affine type $A_{n-1}^{(1)}$; it is the unital associative $\BC(q)$-algebra generated by
$e_i, f_i, K_i^{\pm}$ ($0\leq i < n$), subject to the following relations:

{\allowdisplaybreaks
\begin{gather*}
  K_i^+ K_i^- = 1 = K_i^- K_i^+, \qquad K_i^+ K_j^+ = K_j^+ K_i^+, \\
  K_i^+ e_j K_i^- = q^{a_{ij}}e_j, \qquad K_i^+ f_j K_i^- = q^{-a_{ij}} f_j, \\
  e_if_j - f_je_i = \delta_{ij} \frac{K_i^+ - K_i^-}{q - q^{-1}}, \\
  \sum_{k=0}^{1-a_{ij}} (-1)^k \left[ \begin{smallmatrix} 1-a_{ij} \\ k \end{smallmatrix} \right]_q e_i^{1-a_{ij}-k} e_j \, e_i^k = 0 \quad (i \ne j), \\
  \sum_{k=0}^{1-a_{ij}} (-1)^k \left[ \begin{smallmatrix} 1-a_{ij} \\ k \end{smallmatrix} \right]_q f_i^{1-a_{ij}-k} f_j \, f_i^k = 0 \quad (i \ne j).
\end{gather*}}
Here $A = (a_{ij})_{0 \leq i,j < n}$ is the Cartan matrix of type $A_{n-1}^{(1)}$, and $\left[ \begin{smallmatrix} m \\ k \end{smallmatrix} \right]_q = \frac{[m]_q[m-1]_q \dotsm [m-k+1]_q}{[k]_q[k-1]_q \dotsm [1]_q}$ where $[i]_q = q^{1-i}+q^{3-i}+\dotsb + q^{i-3} + q^{i-1}$ for all $i \in \BN$.

For convenience, if $j \in \mathbb{Z}$ and $j \equiv_n \bar{j}$ with $\bar{j} \in \{0,1,\dotsc, n-1\}$, we also write $f_j$ for $f_{\bar{j}}$.

Fix a tuple $(n_1,\ldots, n_r)$ of positive integers such that $n_1+\cdots+n_r = n$, and, for each $j= 1,\dotsc,r$, let $$\sigma_j = \sum_{a=1}^j n_a.$$
Let $\Up$ be the corresponding parabolic subalgebra of $\BU$, isomorphic to $U_q(\Gsl_{n_1})\otimes \dotsb \otimes U_q(\Gsl_{n_r})$, the quantized enveloping algebra of finite type
$A_{n_1-1} \times \dotsb \times A_{n_r-1}$. So $\Up$ is the subalgebra of $\U$ generated by the $e_i$, $f_i$ and $K_i^{\pm}$ such that $i\neq \sigma_k$ for all $k$ with $0\leq k \leq r-1$.

The assignments $\phi(e_i)=e_i$,
$\phi(f_i)=f_i$, $\phi(K_i^{\pm})=K_i^{\mp}$
%and $\overline{q}=q^{-1}$
define a $\BC(q)$-semilinear automorphism
$\phi:\U\to\U$. (Here, and hereafter, a map $f : V \to W$ between $\BC(q)$-vector spaces is $\BC(q)$-semilinear if and only if $f(v_1 + a(q)v_2) = f(v_1) + a(q^{-1})f(v_2)$ for all $v_1,v_2 \in V$ and $a(q) \in \BC(q)$.) Likewise,
$\omega(e_i)=f_i$, $\omega(f_i)=e_i$,
$\omega(K_i^{\pm})=K_i^{\pm}$ % and $\omega(q)=q$
can be extended to
a $\BC(q)$-linear antiautomorphism of $\U$.  The maps $\phi$ and $\omega$ restrict to an automorphism and an antiautomorphism of $\Up$, respectively.

We define a coproduct $\Delta : \U\to\U\otimes\U$ by
\begin{align*}
\Delta(e_i) &= e_i\otimes 1 + K^{-}_i \otimes e_i, \\
\Delta(f_i) &= f_i\otimes K_i^+ + 1\otimes f_i, \\
\Delta(K^{\pm}_i) &= K^{\pm}_i \otimes K^{\pm}_i.
\end{align*}
It restricts to a coproduct on $\Up$.
Note that $\Delta$ differs from the coproduct used in \cite{Lu} and \cite{Jan}, which is $(\phi\otimes\phi)\circ\Delta\circ\phi$, and from that used in \cite{Ka}, \cite{HK} and \cite{Le}, which is $(\omega\otimes\omega)\circ\Delta\circ\omega$.

Given two $\U$-modules (resp. $\Up$-modules) $M$ and $N$, we take the $\U$-module (resp. $\Up$-module) structure on $M\otimes N$ to be given by the pullback along $\Delta$.

\subsection{Fock spaces} \label{prelimFock}
Let
$$\CF:=\bigoplus_{\lambda\in\CP} \mathbb{C}(q) \lambda$$
be the $\mathbb{C}(q)$-vector space with distinguished basis given by the set $\CP$ of partitions of all natural numbers.

In \cite{LT}, Leclerc and Thibon defined an $\BC(q)$-semilinear bar involution $x \mapsto \overline{x}$ on $\mathcal{F}$. They proved the
existence of another distinguished basis $\{ G(\lambda) \mid \lambda \in \mathcal{P} \}$ of $\mathcal{F}$, called the canonical basis, which has the following characterization:
$$ G(\lambda) - \lambda \in \sum_{\mu \in \mathcal{P}} q\mathbb{Z}[q] \mu,\quad \overline{G(\lambda)} = G(\lambda).$$

Let $\left\langle - , - \right\rangle$ be the symmetric bilinear form on $\mathcal{F}$ with respect to which $\mathcal{P}$ is orthonormal.
For $\lambda,\mu \in \mathcal{P}$, define $d_{\lambda\mu}(q) \in \mathbb{C}(q)$ by
$$
d_{\lambda\mu}(q) = \langle G(\mu), \lambda \rangle.
$$
As the involution defined in \cite{LT} depends on a fixed integer $n \geq 2$, $d_{\lambda \mu}(q)$ also implicitly depends on $n$.  When we need to emphasize the role of $n$, we shall write $d^{\,n}_{\lambda\mu}(q)$ instead.  For convenience, we further define $d^{\,1}_{\lambda\mu}(q) = \delta_{\lambda\mu}$.

We shall make use of the following remarkable properties of these $q$-decomposition numbers $d_{\lambda\mu}(q)$:

\begin{thm} \label{T:appendix}
\hfill
\begin{enumerate}
\item $d_{\mu\mu}(q) = 1$.
\item If $\lambda \ne \mu$, then $d_{\lambda\mu}(q) \in q\mathbb{N}_0[q]$.
\item If $d_{\lambda\mu}(q) \ne 0$, then $\mu \geq_{J_n} \lambda$.
\end{enumerate}
\end{thm}

Recall the collection $\B$ of all sets of $\beta$-numbers.   Define $\CF_{\BZ}$ to be the $\BC(q)$-vector space with basis $\B$. Following \cite{H,MM}, we define an action of $\U = U_q(\widehat{\Gsl}_n)$ on $\CF_{\BZ}$ as follows.  Let $B \in \B$, and take $x\notin B$ such that $x-1 \in B$.  Let $C = B \setminus \{x-1\} \cup \{x\}$, and
\begin{align*}
N_>(B,C) &= |\{ y \in B \mid y > x,\ y \equiv_n x-1 \}| - |\{ y \in B \mid y>x,\ y \equiv_n x \}|, \\
N_<(B,C) &= |\{ y \notin B \mid y < x-1,\ y \equiv_n x-1 \}| - |\{ y \notin B \mid y<x-1,\ y \equiv_n x \}|.
\end{align*}
Here, and hereafter, we write $a\equiv_n b$ for $a \equiv b \pmod n$. For $B\in\B$ and $0\leq i < n$, let
$$N_i(B)=|\{y\notin B\mid y-1\in B, y\equiv_n i\}|
-
|\{y\in B\mid y-1\notin B, y\equiv_n i\}|.$$
Then we have
\begin{align*}
e_i(C) = \sum_B q^{N_<(B,C)}C,  \qquad   f_i(B) = \sum_{C} q^{N_>(B,C)}C, \qquad
K^{+}_i(B)=q^{N_i(B)},
\end{align*}
where the first sum runs over all $B\in \B$ such that $B = C\setminus \{x\} \cup \{x-1\}$ for some $x \in C$, $x \equiv_n i$ and $x-1 \notin C$, and second sum over all $C\in\B$ such that $C = B \setminus \{x-1\} \cup \{x\}$ for some $x-1 \in B$, $x \equiv_n i$ and $x \notin B$.

Let $s\in \BZ$, and write $\CF_s$ for the vector subspace of $\CF_{\BZ}$ with basis $\B_s$ ($= \{ B \in \B \mid \s(B) = s\}$).  It is easy to see that $\CF_s$ is invariant under the $\U$-action defined above.  Furthermore, the bijection from $\CP$ to $\B_s$ defined by $\lambda \mapsto \beta_s(\lambda)$ induces a $\BC(q)$-linear isomorphism $\beta_s : \CF \to \CF_s$.  Thus via this isomorphism, $\CF_s$ inherits the bar-involution from $\CF$, while $\CF$ inherits an $\U$-action from $\CF_s$.  The action of $e_i$ and $f_i$ on $\CF$ via $\beta_s$ may be described as follows.  Let $i_s$ be the residue class of $i-s$ modulo $n$. Suppose that $\lambda \in \CP$ has an addable node $\mathfrak{n}$ of $n$-residue $i_s$, and let $\mu$ be the partition obtained by adding $\mathfrak{n}$ to $\lambda$.  Let $N_>(\lambda,\mu)$ (resp.\ $N_<(\lambda,\mu)$) be the number of addable nodes of $\lambda$, of $n$-residue $i_s$ and to the right (resp.\ left) of $\mathfrak{n}$, minus the number of removable nodes of $\lambda$, of $n$-residue $i_s$ and to the right (resp.\ left) of $\mathfrak{n}$.  For $\lambda\in\CP$ and $0\leq i<n$, let
$N_i(\lambda)$ be the number of addable nodes of $\lambda$ of $n$-residue $i$ minus the number of removable nodes of $\lambda$ of $n$-residue $i$. Then
$$
e_i(\mu) = \sum_{\lambda} q^{-N_<(\lambda,\mu)} \lambda, \qquad f_i(\lambda) = \sum_{\mu} q^{N_>(\lambda,\mu)} \mu,\qquad
K^{+}_i(\lambda)=q^{N_i(\lambda)}\lambda,
$$
where the first sum runs over all partitions $\lambda$ that can be obtained by removing a removable node of $n$-residue $i_s$ from $\mu$ and the second sum runs over all partitions $\mu$ that can be obtained by adding an addable node of $n$-residue $i_s$ to $\lambda$.

For each $\lambda \in \CP$, write $G_s(\lambda)$ for $\beta_s(G(\lambda))$.
%Now we
\begin{thm}[\cite{LT}] \label{T:appendixlast}
For all $u\in \U$ and $x\in\mathcal{F}_s$, we have
$$\overline{ux}=\phi(u)\,\overline{x}.$$
In particular, for $0\leq i < n$,
$$f_i(G_s(\lambda)) = \sum_{\mu \in \CP} L_\lambda^{\mu}(q) G_s(\mu),$$
where $L_\lambda^{\mu}(q) = L_\lambda^{\mu}(q^{-1}) \in \BC(q)$ for all $\mu \in \CP$.
%Moreover $L_\lambda^{\mu}(q) = \mathbb{N}_0[q,q^{-1}]$ for all $\mu \in \CP$.
\end{thm}
Lascoux and Thibon in fact extend the action of $\U$ on $\mathcal{F}_s$ to an action of the larger algebra
$U_q(\widehat\Ggl_{n})$, and prove the theorem in this context. We will not require the extended action.

\subsection{Tensor products of based modules}
In this subsection we give an abbreviated introduction to Lusztig's theory of based modules over quantized enveloping algebras; see \cite[Chapters 27-28]{Lu} for full details, though note that since our coproduct is slightly different, we need to swap $q$ and $q^{-1}$ in Lusztig's account. The theory is valid in finite type; we consider specifically based modules of $\Up$, the subalgebra of $\U$ defined in \autoref{qea}.

A based $\Up$-module $M$ is a finite-dimensional weight $\Up$-module equipped with a $\BC(q)$-basis $B$ satisfying certain conditions, amongst which are the following:
\begin{enumerate}
 \item each $b\in B$ is a weight vector.
 \item for all $u\in \Up$ and $m\in M$, we have $\psi_M(um)=\phi(u)\psi_M(m)$,
where $\psi_M$ is the $C(q)$-semilinear involution of $M$ fixing every element of $B$.
\end{enumerate}

Any highest weight irreducible $\Up$-module, together with its canonical basis, is a based module. %\marginpar{in the quantum case is it still called highest weight module? check.}

Recall that $\Up\cong U_q(\Gsl_{n_1})\otimes\dotsb
\otimes U_q(\Gsl_{n_r})$. Given a based
$U_q(\Gsl_{n_j})$-module $(M^{(j)},B^{(j)})$ for each
$j=1,\dotsc, r$, we may form the `external' tensor product
$$M=M^{(1)}\otimes\dotsb\otimes M^{(r)},$$
a $\Up$-module. It is then clear that the basis
$B=B^{(1)}\otimes\dotsb\otimes B^{(r)}$ satisfies the two properties
above, and in fact it is true that $(M,B)$ is a based
$\Up$-module.

We now turn to the consideration of `internal' tensor products.
Let $(M_1,B_1),\ldots, (M_d,B_d)$ be based modules of $\Up$. The coproduct on $\Up$ gives a $\Up$-module structure
$M=M_1\otimes\dotsb\otimes M_d$. However $M$ together with
the obvious basis
$\BB^{\otimes} := B_1\otimes \dotsb\otimes B_d$ is not necessarily a based module. In order to correct this deficiency,
Lusztig defines an involution $\psi_M$ of $M$, constructed out of the involutions $\psi_{M_i}$ and the quasi-$R$-matrix
of $\Up$, and shows that for each
$\Bb = (b_1,\dotsc,b_d)\in \BB := B_1\times\dotsb\times B_d$
there exists a unique element
$$\Bb^{\tp}=b_1\tp\dotsb\tp b_d \in M$$
such that
%(Really, iterated application of \cite[Theorem 27.3.2(a)(c)]{Lu}.)
$$\psi_M(\Bb^{\tp}) =\Bb^{\tp} \quad \text{and}\quad \Bb^{\tp}-\Bb^{\otimes} \in \bigoplus_{\Bc\in \BB} q\BZ[q]\Bc^{\otimes},$$
where $\Bb^{\otimes} = b_1 \otimes \dotsb \otimes b_d$.
He also proves that, writing
$$
\BB^{\tp} = B_1 \tp \dotsb \tp B_d := \{ \Bb^{\tp} \mid \Bb \in \BB\},$$
$(M, \BB^{\tp})$ is a based $\Up$-module.

In order to formulate some required additional properties of this canonical basis,
we introduce a partial order on $\BB$, a reverse lexicographic order, as follows:
$\Bb = (b_1,\dotsc,b_d) > \Bb' = (b'_1,\dotsc,b'_d)$ if and only if there exists $i$ such that $\wt(b_j)=\wt(b'_j)$ for all $i<j\leq d$ and $\wt(b_i)< \wt(b'_i)$.

\begin{lemma}\label{l:canonicaltensorproduct}
Keep the notations above.
\begin{enumerate}
\item We have, for all $\Bb\in \BB$,
$$\Bb^{\tp}-\Bb^{\otimes} \in \bigoplus_{\Bc<\Bb} q\BZ[q]\Bc^{\otimes}.$$
%\begin{align*}
%\Bb^{\tp}-\Bb &\in \bigoplus_{\Bc<\Bb} q\BZ[q]\Bc,\\
%\Bb^{\tm}-\Bb &\in \bigoplus_{\Bc<\Bb} q^{-1}\BZ[q^{-1}]\Bc.\\
%\end{align*}
\item Let $\Bb = (b_1,\dotsc,b_d)\in \BB$, and suppose that $b_d$ is a highest weight vector in
$M_d$. %\marginpar{Do we need $M_d$ to be a highest weight module with highest weight vector $b_d$?}
Writing $\Ba=(b_1,\dotsc,b_{d-1})$, we have
$$\Bb^{\tp}=\Ba^{\tp}\otimes b_d.$$
\end{enumerate}
\end{lemma}

\begin{proof}
This first part follows by iterated application of
\cite[Theorem 27.3.2(b)]{Lu}. %Note that our partial order is finer than Lusztig's.
 The second is checked by directly evaluating
$\psi_M(\Ba^{\tp} \otimes b_d)$; see the proof of \cite[Lemma 28.2.6]{Lu}
\end{proof}

Finally we explain how to handle a mixture of external and internal tensor products of based modules.
Suppose that for $j=1,\dotsc,r$ we are given based $U_q(\Gsl_{n_j})$-modules $$(M_1^{(j)},B_1^{(j)}),\dotsc,
(M_d^{(j)},B_d^{(j)}).$$ Then, as described above, for each $j$ we have a
based $U_q(\Gsl_{n_j})$-module
$$M^{(j)}=M_1^{(j)}\otimes\dotsb\otimes M_d^{(j)}$$
with basis
$$\BB^{(j)\tp} = B^{(j)}_1 \tp \dotsb \tp B^{(j)}_d.$$
We can then take the external tensor product, obtaining
a based $\Up$-module
$$M'=M^{(1)}\otimes\dotsb\otimes M^{(r)}$$ with
basis
$$\BB^{\tp\otimes}=\BB^{(1)\tp}\otimes\dotsb\otimes \BB^{(r)\tp}.$$

On the other hand, we may take external tensor products first, before taking internal tensor products. For each $i=1,\ldots, d$, we have the based $\Up$-module
$$ M_i=M_i^{(1)}\otimes\dotsb\otimes M_i^{(r)}.$$
with basis
$$\BB_i^{\otimes}=B_i^{(1)}\otimes\dotsb\otimes B_i^{(r)}.$$
The internal tensor product
$$M''=M_1\otimes\dotsb\otimes M_d$$
is then a based $\Up$-module with basis
$$\BB^{\otimes \tp} =\BB_1^{\otimes}\tp\dotsb\tp \BB_d^{\otimes}.$$

Thus, to each $\Bb \in \BB = \prod_{1\leq i \leq d,\, 1 \leq j \leq r} B^{(j)}_i$, we have associated a distinguished basis element $\Bb^{\tp \otimes} \in \BB^{\tp \otimes} \subseteq M'$ and another distinguished basis element $\Bb^{\otimes \tp} \in \BB^{\otimes \tp} \subseteq M''$.

\begin{prop}\label{P:internal-external}
Keep the notation as above. The obvious correspondence between $M'$ and $M''$ gives an isomorphism of based $\Up$-modules.  Furthermore $\Bb^{\tp \otimes} \leftrightarrow \Bb^{\otimes \tp}$ for each $\Bb \in \BB$ under this correspondence.
\end{prop}

\begin{proof}
Since the coproduct of $\Up$ is simply the tensor product
of the coproducts of its factors $\U_q(\Gsl_{n_j})$, the
natural identification of vector spaces
$$M=M^{(1)}\otimes\dotsb\otimes M^{(r)} \,=\enspace
%\bigotimes_{(i,j)\in[1,d]\times[1,r]}
\bigotimes_{\substack{1\leq i \leq d \\ 1\leq j \leq r}}
 M_i^{(j)}\,=\enspace
M_1\otimes\dotsb\otimes M_d=M''$$
is an isomorphism of $\Up$-modules.
Under this identification $\Bb^{\tp \otimes} = \Bb^{\otimes \tp}$ for all $\Bb \in \BB$, as they are characterized in terms of involutions that coincide because the quasi-$R$-matrix of $\Up$ is the tensor product of the
quasi-$R$-matrices of its factors $U_q(\Gsl_{n_j})$.
\end{proof}

\subsection{Schur algebras}
Let $\BS_m = S(m,m)$ be
the Schur algebra of degree $m$ over a field $\mathbb{F}$ of characteristic $p>0$. This is a quasi-hereditary algebra with indexing set $\{ \lambda\in\CP \mid |\lambda|=m\}$, together with the order opposite to the lexicographic order; here we are following the notation
conventions in \cite{Jam}.  Denote by $L(\lambda)$, $\Delta(\lambda)$ and $P(\lambda)$
the simple, standard and projective indecomposable modules associated to $\lambda$. By Brauer-Humphrey's reciprocity the decomposition number $d_{\lambda\mu}:=[\Delta(\lambda):L(\mu)]$ may be expressed as the multiplicity
$[P(\mu):\Delta(\lambda)]$ of $\Delta(\lambda)$ in a $\Delta$-filtration of $P(\mu)$; even though the filtrations are not unique, the multiplicities are.  These numbers enjoy the following well-known properties:

\begin{prop} \label{decompnumbers}
Let $\lambda, \mu \in \CP$.  Then
\begin{enumerate}
\item $d_{\lambda \lambda} = 1$;
\item $d_{\lambda\mu}\neq 0$ only if $\lambda\leq_{J_p} \mu$ (see \cite[Theorem 2.5(iv)]{TT});
\item $d_{\lambda\mu}\geq d_{\lambda\mu}(1)$ (by \cite{VV} and \cite{Jam}).
\end{enumerate}
\end{prop}

For each residue class $i$ modulo $p$, one can define functors
\begin{align*}
i\text{-Res}_m &: \BS_{m+1}\text{-\textbf{mod}} \to \BS_m\text{-\textbf{mod}}, \\
i\text{-Ind}_{m} &: \BS_m\text{-\textbf{mod}} \to \BS_{m+1}\text{-\textbf{mod}}.
\end{align*}
They are a biadjoint pair of exact functors, sending projectives to projectives.  In particular, if $\lambda \in \CP$ with $|\lambda| = m$, then
$$ i\text{-Ind}_{m} (P(\lambda)) \cong \bigoplus_{\substack{\mu \in \CP \\ |\mu| = m+1}} P(\mu)^{\oplus L_{\lambda}^\mu},
$$
where each $L_{\lambda}^{\mu} \in \BN_0$.  In fact,
$$
L_{\lambda}^{\mu} = \dim_{\mathbb{F}}( \Hom_{\BS_{m+1}}(i\text{-Ind}_{m} (P(\lambda)), L(\mu)))
= \dim_{\mathbb{F}}(\Hom_{\BS_m}(P(\lambda), i\text{-Res}_{m} (L(\mu)))),
%= [r\text{-Res}_{m} (L(\mu)): L(\lambda)],
$$
so that $L_{\lambda}^{\mu}$ is the multiplicity $[i\text{-Res}_{m} (L(\mu)): L(\lambda)]$ of $L(\lambda)$ as a composition factor of $i\text{-Res}_{m} (L(\mu))$.  Kleshchev describes some of these multiplicities.

\begin{theorem}[{\cite[Theorem 9.3]{Kl}}] \label{Klthm}
Let $\mu \in \CP$ with $|\mu| = m$.  Suppose that $\mu$ has a normal node $\mathfrak{n}$ of $p$-residue $i$, and let $\lambda$ be the partition obtained when this node is removed from $\mu$.  Then
\begin{multline*}
[i\text{-}{\rm Res}_{m} (L(\mu)): L(\lambda)] = \\1 + \text{number of normal nodes of $\mu$ of $p$-residue $i$ and to the right of $\mathfrak{n}$}.
\end{multline*}
\end{theorem}

Let $i\text{-Res} = \bigoplus_{m} i\text{-Res}_m$ and $i\text{-Ind} = \bigoplus_m i\text{-Ind}_m$.
For convenience, for $k \in \mathbb{Z}$, we also write $k\text{-Res}$ and $k\text{-Ind}$, which mean $\bar{k}\text{-Res}$ and $\bar{k}\text{-Ind}$ respectively, where $\bar{k}$ is the residue class of $k$ modulo $p$.

The effects of $i$-Res and $i\text{-Ind}$ on the standard module $\Delta(\lambda)$ can be easily described in the Grothendieck group $K_0(\bigoplus_m \BS_m\text{-\textbf{mod}})$ as follows:
\begin{align*}
[i\text{-\textrm{Res}}(\Delta(\lambda))] &= \sum_{\tau} [\Delta(\tau)]; \\
[i\text{-\textrm{Ind}}(\Delta(\lambda))] &= \sum_{\mu} [\Delta(\mu)],
\end{align*}
where the first sum runs over all partitions $\tau$ that can be obtained by removing a removable node of $p$-residue $i$ from $\lambda$, while the second sum runs over all partitions $\mu$ that can be obtained by adding an addable node of $p$-residue $i$ to $\lambda$.

It is natural to identify $K_0(\bigoplus_m \BS_m\text{-\textbf{mod}})$ with the classical (non-quantized) Fock space $\CF_c= \bigoplus_{\lambda \in \CP} \BZ \lambda$ under the correspondence $[\Delta(\lambda)] \leftrightarrow \lambda$.

Recall the quantized Fock space $\CF$ introduced in \autoref{prelimFock}.  Let $A = \BZ[q,q^{-1}]$, and let $\CF_{A}$ be the free $A$-submodule of $\CF$ with basis $\CP$.  Then $G(\lambda) \in \CF_{A}$ for all $\lambda \in \CP$ by \autoref{T:appendix}.
Similarly, let $\CF_{s,A}$ denote the free $A$-submodule of $\CF_{s}$ with basis $\B_s$ for each $s \in \BZ$.  We have a surjective $\BZ$-linear map $\E : \CF_{\BZ,A} \to \CF_c$ defined by $a(q)B \mapsto a(1)\Par(B)$ for all $a(q) \in A$ and $B \in \B$.

Let $n= p$, and let $\U_A$ be the $A$-subalgebra of $\U$ generated by $\{ e_i, f_i, K_i^{\pm} \mid i = 0 ,1,\dotsc, p-1\}$.  Then for each $s \in \BZ$, $\CF_{s,A}$ is a $\U_A$-submodule of $\CF_s$.  Under the above identification of $K_0(\bigoplus_m \BS_m\text{-\textbf{mod}})$ with $\CF_c$, we have that $\E$, when restricted to $\CF_{s,A}$, intertwines the Chevalley generators $e_i$ and $f_i$ with the $(i-s)$-Res and $(i-s)$-Ind functors.  More precisely, we have
$$
\E \circ f_i  = (i-s)\text{-}\mathrm{Ind} \circ \E \qquad \text{and}\qquad
\E \circ e_i  = (i-s)\text{-}\mathrm{Res} \circ \E.
$$

\section{General setup} \label{setup}
In order to deal with the quantized Fock spaces and the classical non-quantized Fock space simultaneously, we consider the following general setup.

Let $R$ be an integral domain, and let $\mathfrak{F}$ be the free $R$-module with $R$-basis $\{ s(\lambda) \mid \lambda \in \CP \}$, indexed by the set $\CP$ of all partitions.  Let $\langle\, , \rangle$ denote the symmetric $R$-bilinear form on $\mathfrak{F}$ with respect to which $\{ s(\lambda) \mid \lambda \in \CP \}$ is orthonormal.

For each $\lambda \in \CP$, let $H(\lambda) \in \mathfrak{F}$ be such that
$\langle H(\lambda), s(\mu) \rangle \ne 0$ only if $\lambda \geq_{J_n} \mu$, and $\langle H(\lambda), s(\lambda) \rangle = 1$.  Then $\{ H( \lambda) \mid \lambda \in \CP \}$ is another $R$-basis of $\mathfrak{F}$.  We note the following easy lemma:

\begin{lemma} \label{easy}
Let $x = \sum_{\mu \in \CP} a_{\mu} H(\mu) \in \mathfrak{F}$, where $a_{\mu} \in R$ for all $\mu \in \CP$.  Let $\nu \in \CP$ be such that $a_{\mu} = 0$ whenever $\mu >_{J_n} \nu$. Then $\langle x, s(\nu) \rangle = a_{\nu}$.
\end{lemma}

\begin{proof}
We have $a_{\mu} = 0$ if $\mu >_{J_n} \nu$, while $\langle H(\mu), s(\nu) \rangle = 0$ if $\mu \not\geq_{J_n} \nu$, so that
$$ \langle x, s(\nu) \rangle = \langle \sum_{\mu \in \CP} a_{\mu} H(\mu), s(\nu) \rangle = \langle a_{\nu}H(\nu),s(\nu) \rangle = a_{\nu}.$$
\end{proof}

Fix $n \in \BN$ and let $\mathfrak{r}$ be a residue class modulo $n$.  Let $\mathfrak{f}_{\mathfrak{r}} : \mathfrak{F} \to \mathfrak{F}$ be an $R$-linear map such that $\langle \mathfrak{f}_{\mathfrak{r}}(s(\lambda)), s(\mu) \rangle \ne 0$ only if $\lambda$ can be obtained from $\mu$ by removing a removable node of $n$-residue $\mathfrak{r}$, and equals $1$ if $\lambda$ is obtained by removing the rightmost such node which also lies to the right of all addable nodes of $n$-residue $\mathfrak{r}$.

\begin{prop} \label{geninductivelemma}
Let $\lambda \in \CP$, and suppose that its rightmost addable node of $n$-residue $\mathfrak{r}$ lies to the right of all its removable nodes of $n$-residue $\mathfrak{r}$, and let $\mu$ be the partition obtained by adding this node.  Then
$$
\mathfrak{f}_{\mathfrak{r}}(H(\lambda)) = H(\mu) + \sum_{\substack{\nu \in \CP \\ |\nu| = |\mu| \\ \nu < \mu}} a_{\nu}H(\nu),
$$
with $a_{\nu} \in R$ for all $\nu$.
\end{prop}

\begin{proof}
Note first that when we remove from $\mu$ its rightmost removable node of $n$-residue $\mathfrak{r}$, we obtain $\lambda$.

Since $\{ H(\nu) \mid \nu \in \CP\}$ is a basis for $\mathfrak{F}$, we have
$\mathfrak{f}_{\mathfrak{r}}(H(\lambda)) = \sum_{\nu \in \CP} a_{\nu}H(\nu)$, where $a_{\nu} \in R$ for all $\nu \in \CP$.  Since $\left< H(\rho), s(\sigma)\right> \ne 0$ only if $\rho \geq_{J_n} \sigma$, and the latter only if $|\rho| = |\sigma|$, we see that $a_{\nu} = 0$ unless $|\nu| = |\lambda| +1 = |\mu|$. Thus, we need to show that $a_{\mu} = 1$, and $a_{\nu} \ne 0$ only if $\nu \leq \mu$.

Let $\sigma \in \CP$ be maximal (with respect to $\geq$) subject to $a_{\sigma} \ne 0$.  Then, by \autoref{easy}, we have
\begin{align*}
(0 \ne)\: a_{\sigma} = \langle \mathfrak{f}_{\mathfrak{r}}(H(\lambda)), s(\sigma) \rangle
 &= \langle \mathfrak{f}_{\mathfrak{r}} \left(\sum_{\rho \in \CP} \langle H(\lambda),s(\rho) \rangle\, s(\rho)\right), s(\sigma) \rangle \\
 &= \sum_{\rho \in \CP} \langle H(\lambda),s(\rho) \rangle \langle \mathfrak{f}_{\mathfrak{r}}(s(\rho)), s(\sigma) \rangle.
\end{align*}
Thus there exists $\rho \in \CP$ such that $\langle H(\lambda), s(\rho) \rangle, \langle \mathfrak{f}_{\mathfrak{r}}(s(\rho)), \sigma \rangle \ne 0$; in particular, $\rho \leq \lambda$, $|\rho| = |\lambda|$, and $\rho$ can be obtained from $\sigma$ by removing a removable node of $n$-residue $\mathfrak{r}$.  This also shows that $|\sigma| = |\mu|$.

If $\rho = \lambda$, then $\sigma$ can be obtained by adding an addable node of $n$-residue $\mathfrak{r}$ to $\lambda$, so that since $\mu$ is obtained by adding the rightmost addable node of $\lambda$ of $n$-residue $\mathfrak{r}$, we have $\mu \geq \sigma$ (in fact, $\mu \geq_{J_n} \sigma$).

If $\rho \ne \lambda$, then $\rho < \lambda$.  Let $a,b \in \BN$ be such that
$\mu_i = \lambda_i + \delta_{ai}$, $\sigma_i = \rho_i + \delta_{bi}$ for all $i \in \BN$.  Also, let $k \in \BN$ be such that $\rho_i = \lambda_i$ for all $i< k$ and $\rho_k < \lambda_k$.
%Then $b \geq \min(a,k)$: otherwise $\lambda_{b+1} \leq \rho_{b+1} \leq \rho_b < \sigma_b = \rho_b+1 = \lambda_b+1$, so that $(b,\sigma_{b}) = (b,\mu_b+1)$ is an addable node of $\mu$ of $n$-residue $\mathfrak{r}$ on the right of $(a,\lambda_a+1)$
We consider the cases $a < k$, $a=k$ and $a>k$ separately.
\begin{description}
\item[$a<k$] Note first that in this case $a \leq b$: otherwise $\lambda_i = \rho_i$ for all $i \leq b$ while $\sigma_b = \rho_b+1$, so that $(b,\rho_b+1)=(b,\lambda_b+1)$ is an addable node of $\rho$ and $\lambda$ of $n$-residue $\mathfrak{r}$ lying to the right of $(a,\lambda_a+1)$, contradicting the latter being the rightmost addable node of $\lambda$ of $n$-residue $\mathfrak{r}$.  Thus $\mu_i = \lambda_i = \rho_i = \sigma_i$ for all $i < a$, and $\mu_a = \lambda_a+1 = \rho_a + 1 = \sigma_a + 1 - \delta_{ba}$.  Consequently either $\mu_a > \sigma_a$ or $a=b$; in the latter case $\mu_i = \lambda_i + \delta_{ai} = \rho_i + \delta_{bi} = \sigma_i$ for all $i < k$, and $\mu_k = \lambda_k > \rho_k = \sigma_k$.  Either way, we have $\mu > \sigma$.
\item[$a=k$] In this case, we have $\mu_i = \lambda_i = \rho_i$ for all $i < a = k$, and $\mu_a = \lambda_a +1 = \lambda_k + 1 > \rho_k + 1 = \rho_a + 1 = \sigma_a + 1 - \delta_{ba} \geq \sigma_a$.  Thus $\mu > \sigma$.
\item[$a>k$] In this case, we have $\mu_i = \lambda_i = \rho_i$ for all $i < k$, and $\mu_k = \lambda_k > \rho_k = \sigma_k - \delta_{bk}$.  Thus either $\mu_k > \sigma_k$ or both $b=k$ and $\mu_k = \sigma_k$.  In the latter case, the node $(k,\sigma_k) = (k,\mu_k)$ is a removable node of $\sigma$ whose removal produces $\rho$, so that it has $n$-residue $\mathfrak{r}$; but this cannot be a removable node of $\mu$ since $(a,\mu_a)$ is its rightmost removable node of $n$-residue $\mathfrak{r}$.  As such, $(k+1,\mu_k) \in [\mu] \setminus [\sigma]$.  This gives $\mu_{k+1} > \sigma_{k+1}$.  Thus, once again, either way, we get $\mu > \sigma$.
\end{description}

We thus conclude that $\mu \geq \sigma$ always.  To finish off the proof, it suffices to show that $a_{\mu} = 1$.
By \autoref{easy}, we have, as above,
$$
a_{\mu} = \langle \mathfrak{f}_{\mathfrak{r}}(H(\lambda)), s(\mu) \rangle = \sum_{\rho \in \CP} \langle H(\lambda),s(\rho) \rangle \langle \mathfrak{f}_{\mathfrak{r}}(s(\rho)), s(\mu) \rangle.
$$
For each $\rho$ such that $\langle \mathfrak{f}_{\mathfrak{r}}(s(\rho)), s(\mu) \rangle \ne 0$, $\rho$ must be obtained from $\mu$ by removing a removable node of $n$-residue $\mathfrak{r}$.  Since $\lambda$ is obtained by removing from $\mu$ its rightmost removable node of $n$-residue $\mathfrak{r}$, we see that $\lambda \leq \rho$ (in fact, $\lambda \leq_{J_n} \rho$), so that $\langle H(\lambda),s(\rho)\rangle = 0$ unless $\lambda = \rho$.  Thus,
$$
a_{\mu} = \langle H(\lambda),s(\lambda) \rangle \langle \mathfrak{f}_{\mathfrak{r}}(s(\lambda)), s(\mu) \rangle = 1.
$$
\end{proof}

\begin{rem}
\autoref{geninductivelemma} in fact holds even when $H(\lambda)$ satisfies the weaker condition that $\left< H(\lambda), s(\mu) \right> \ne 0$ only if $\lambda > \mu$ and $|\lambda| = |\mu|$.  The proof is similar, uses a modified version of \autoref{easy}.  We leave it to the reader as an easy exercise.
\end{rem}

Applying \autoref{geninductivelemma} to $\CF_c$ ($\cong K_0(\bigoplus_m \mathbf{S}_m\text{-\textbf{mod}})$) and $\CF_{A,s}$ in the place of $\mathfrak{F}$, and $i$-Ind and $f_{s+i}$ in the place of $\mathfrak{f}$, we get the following immediately corollary.

\begin{cor} \label{inductivecor}
Let $\lambda \in \CP$, and suppose that its rightmost addable node of $n$-residue $i$ lies to the right of all its removable nodes of $n$-residue $i$, and let $\mu$ be the partition obtained by adding this node.  Then
\begin{enumerate}
\item when $n= p$,
$$i\text{-}\mathrm{Ind}(P(\lambda)) \cong P(\mu) \oplus \bigoplus_{\substack{\nu \in \CP \\ |\nu| = |\mu| \\ \nu < \mu}} P(\nu)^{\oplus [i\text{-}\mathrm{Res} (L(\nu))\, :\, L(\lambda)]};$$
\item
$$f_{s+i}(G_s(\lambda)) = G_s(\mu) + \sum_{\substack{\nu \in \CP \\ |\nu| = |\mu| \\ \nu < \mu}} L^{\nu}_{\lambda}(q)\, G_s(\nu),$$
where $L^{\nu}_{\lambda}(q) \in \BZ[q,q^{-1}]$ for all $\nu$.
\end{enumerate}
\end{cor}

Suppose further that $\CP = \bigcup_{\Bt \in \BT} \CP_{\Bt}$ can be partitioned into subsets which are indexed by a set $\BT$, and that the indexing set $\BT$ can be partially ordered by $\geq$ so that if $\lambda \in \CP_{\Bt},\ \mu \in \CP_{\Bs}$ and $\lambda \geq_{J_n} \mu$, then $\Bt \geq \Bs$.  For each $\Bt \in \BT$, let $\mathfrak{F}_{\Bt}$ be the free $R$-submodule of $\mathfrak{F}$ with $R$-basis $\{ s(\lambda) \mid \lambda \in \CP_{\Bt} \}$.  Then $\mathfrak{F} = \bigoplus_{\Bt \in \BT} \mathfrak{F}_{\Bt}$; let $\pi_{\Bt} : \mathfrak{F} \to \mathfrak{F}$ denote the natural projection onto $\mathfrak{F}_{\Bt}$ via this decomposition.  We have, as immediate consequences of such an order on $\BT$, the following: for $\Bt \in \BT$ and $\lambda \in \CP_{\Bt}$,
\begin{itemize}
\item $H(\lambda) \in \bigoplus_{\Bs \leq \Bt} \mathfrak{F}_{\Bs}$;
\item if $\Bs \not\leq \Bt$, then $\pi_{\Bs}(H(\lambda)) = 0$.
\end{itemize}

\begin{lemma} \label{genind}
Let $\xi : \mathfrak{F} \to \mathfrak{F}$ be an $R$-linear map satisfying $\xi(\mathfrak{F}_{\Bs}) \subseteq \mathfrak{F}_{\Bs}$ for all $\Bs \in \BT$.  Let $\Bt \in \BT$ and $\lambda \in \CP_{\Bt}$.  Suppose that $\xi(H(\lambda)) = \sum_{\mu \in \CP} a_{\mu} H(\mu)$, where $a_{\mu} \in R$ for all $\mu \in \CP$.  Then
\begin{enumerate}
\item $a_{\mu} = 0$ for all $\mu \in \CP_{\Bs}$ with $\Bs \not\leq \Bt$;
\item $\xi (\pi_{\Bt}(H(\lambda)) = \pi_{\Bt}(\xi(H(\lambda))) = \sum_{\mu \in \CP_{\Bt}} a_{\mu} \pi_{\Bt} (H(\mu))$.
\end{enumerate}
\end{lemma}

\begin{proof}
Since $H(\lambda) \in \bigoplus_{\Bs \leq \Bt} \mathfrak{F}_{\Bs}$, we have $\xi(H(\lambda)) \in \bigoplus_{\Bs \leq \Bt} \mathfrak{F}_{\Bs}$.  If there exists $\nu \in \CP_{\Bs}$ with $\Bs \not\leq \Bt$ such that $a_{\nu} \ne 0$, choose $\nu$ to be maximal with respect to $\geq_{J_n}$.  Then by \autoref{easy},
$$
a_{\nu} = \langle \xi(H(\lambda)), s(\nu) \rangle = 0,$$
since $\xi(H(\lambda)) \in \bigoplus_{\Bs \leq \Bt} \mathfrak{F}_{\Bs}$, a contradiction.  Thus, (1) holds, and hence
$\xi(H(\lambda)) = \sum_{\Bs \leq \Bt} \sum_{\mu \in \CP_{\Bs}} a_{\mu} H(\mu)$, so that
\begin{align*}
\pi_{\Bt} ( \xi(H(\lambda))) &= \pi_{\Bt} (\sum_{\Bs \leq \Bt} \sum_{\mu \in \CP_{\Bs}} a_{\mu} H(\mu)) \\
&= \sum_{\Bs \leq \Bt} \sum_{\mu \in \CP_{\Bs}} a_{\mu} \pi_{\Bt}(H(\mu)) \\
&= \sum_{\mu \in \CP_{\Bt}} a_{\mu} \pi_{\Bt}(H(\mu)),
\end{align*}
proving the second equality of (2).  The first equality of (2) follows from the fact that $\xi$ and $\pi_{\Bt}$ commute, since each $\mathfrak{F}_{\Bs}$ is invariant under $\xi$.
\end{proof}

\section{Restriction to parabolic subalgebra} \label{RestrictToU(n)}

Fix $n \in \BN$ and a tuple $\Bn = (n_1,\dotsc,n_r)$ of positive integers such that $n_1+ \dotsb + n_r = n$, and for each $j = 0, 1,\dotsc, r$, let $\sigma_j = \sum_{i=1}^j n_i$.

We wish to consider the restrictions of the action of
$\U$ on the Fock spaces $\CF_s$ to the parabolic subalgebra $\Up$, as defined in \autoref{qea}.
The restrictions have natural direct sum decompositions; to better understand
the summands, consider the example pictured in \autoref{f:exampleabacusblobs}.

\begin{figure}[h]
\begin{tikzpicture}[y=-1cm, scale=0.6, every node/.style={font=\scalefont{.8}}]

%blobs
\def\w{0.4} %blobwidth
\def\h{0.3} %blobheight
\foreach \y in {-2,...,3}
{
\path[rounded corners, fill=lightgray] (-\w,\y-\h) rectangle (3+\w, \y+\h);
\path[rounded corners, fill=lightgray] (4-\w,\y-\h) rectangle (5+\w, \y+\h);
\path[rounded corners, fill=lightgray] (6-\w,\y-\h) rectangle (8+\w, \y+\h);
};

%runner labels and runners
\foreach \x in {0,...,8}
{\node [above] at (\x, -3) {$\x$};
\draw (\x,-2.5) --(\x,3.5);};

%row labels
\foreach \y in {-2,...,3}
{\node [left] at (-1, \y) {$\y$};};

%beads
\foreach \x in
{(0,-2), (1,-2), (2,-2), (3,-2), (4,-2), (5,-2), (6,-2), (7,-2), (8,-2),
 (0,-1), (1,-1), (2,-1), (3,-1), (4,-1), (5,-1), (6,-1), (7,-1), (8,-1),
 (0, 0),         (2, 0), (3, 0),         (5, 0),         (7, 0),
         (1, 1), (2, 1),         (4, 1),                 (7, 1), (8, 1),
 (0, 2),                 (3, 2),                 (6, 2),         (8, 2)
}
{\draw [fill] \x circle [radius=0.2];};

\node [right] at (9,0) {$ v_{4,0}\wedge v_{4,2}\wedge v_{4,3} \enspace\otimes\enspace v_{2,1} \enspace\otimes\enspace v_{3,1}$};
\node [right] at (9,1) {$ v_{4,1}\wedge v_{4,2} \enspace\otimes\enspace v_{2,0} \enspace\otimes\enspace v_{3,1}\wedge v_{3,2}$};
\node [right] at (9,2) {$ v_{4,0}\wedge v_{4,3} \enspace\otimes\enspace v_{3,0}\wedge v_{3,2}$};
\end{tikzpicture}
\caption{A $9$-abacus display of the partition $\lambda=(13,12,10,8,8,8,6,5,5,3,2,1,1)$ and the corresponding
vector in
$\left(\bw{3} V_4 \otimes V_2 \otimes V_3\right)
\otimes \left(\bw2 V_4 \otimes V_2 \otimes \bw2 V_3\right)
\otimes \left(\bw2 V_4 \otimes \bw2 V_3\right)$
}
\label{f:exampleabacusblobs}
\end{figure}
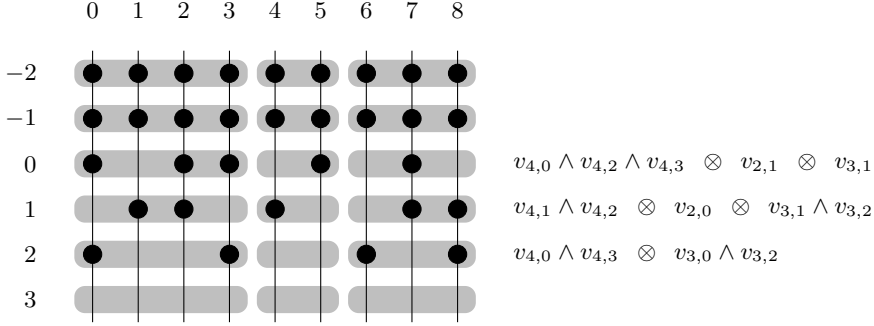

 In view of the description of the action in terms of the
abacus, it is natural to split the $n$-abacus into sections, each of which is represented by a gray region in \autoref{f:exampleabacusblobs}.  Formally, for each $i \in \BZ$ and $j \in \{1,\dotsc,r\}$, let
$$\BZ^{\Bn}_{i,j} = \{ x \in \BZ \mid \sigma_{j-1} \leq x-in < \sigma_{j}\}.$$
Given a subset $B$ of $\BZ$, we write
$$X^{\Bn}_{i,j}(B) = (B \cap \BZ^{\Bn}_{i,j}) - in - \sigma_{j-1} = \{ x-in -\sigma_{j-1} \mid x \in B \cap \BZ^{\Bn}_{i,j}\}.$$
Thus $X^{\Bn}_{i,j}(B)$ describes, in a normalized form, the beads lying in the section $\BZ^{\Bn}_{i,j}$ in the $n$-abacus display of $B$.  In addition, let $\t^{\Bn}_B : \BZ \times \{1,\dotsc,r\} \to \BN_0$ be defined by $\t^{\Bn}_B(i,j) = |B \cap \BZ^{\Bn}_{i,j}| = |X^{\Bn}_{i,j}(B)|$, so that it describes the number of beads lying in each section in the $n$-abacus display of $B$.  Note that $B$ is a set of $\beta$-numbers if and only if
$$
\sum_{\substack{i \geq 0 \\ 1 \leq j \leq r}} \t^{\Bn}_B(i,j) + \sum_{\substack{i < 0 \\ 1 \leq j \leq r}} (n_j - \t^{\Bn}_B(i,j)) < \infty,
$$ in which case
$$
\sum_{\substack{i \geq 0 \\ 1 \leq j \leq r}} \t^{\Bn}_B(i,j) - \sum_{\substack{i < 0 \\ 1 \leq j \leq r}} (n_j - \t^{\Bn}_B(i,j)) = \s(B).
$$

Let $\BT_\Bn$ be the set of functions $\Bt : \BZ \times \{1,\dotsc,r\} \to \BN_0$ such that $\Bt(i,j) \leq n_j$ for all $i \in \BZ$ and that $\sum_{\substack{i \geq 0 \\ 1 \leq j \leq r}} \Bt(i,j) + \sum_{\substack{i < 0 \\ 1 \leq j \leq r}} (n_j - \Bt(i,j)) < \infty$.  For each $\Bt \in \BT_\Bn$, let
$$
\s(\Bt) = \sum_{\substack{i \geq 0 \\ 1 \leq j \leq r}} \Bt(i,j) - \sum_{\substack{i < 0 \\ 1 \leq j \leq r}} (n_j - \Bt(i,j)).
$$
Then for each $B \in \B$, $\s(B) = \s(\t^{\Bn}_B)$.
For each $s \in \BZ$, let $\BT_{\Bn,s} = \{ \Bt \in \BT_\Bn \mid \s(\Bt) = s \}$.

We have a surjection from $\B$ to $\BT_{\Bn}$ defined by $B \mapsto \t^{\Bn}_B$, and this restricts to give a surjection from $\B_s$ to $\BT_{\Bn,s}$ for each $s \in \BZ$.  Given $\Bt \in \BT_\Bn$, let $\B_{\Bt} = \{ B \in \B \mid \t^{\Bn}_B = \Bt \}$.  Then the $\B_{\Bt}$'s partition $\B_s$ as $\Bt$ runs over $\BT_{\Bn,s}$.

Let $\Bt \in \BT_{\Bn}$ and write $s = \s(\Bt)$.  We now describe a characterisation of $B \in \B_{\Bt}$ in terms of $\Par(B)$.
Define $B_{\Bt} \subseteq \mathbb{Z}$ as follows: $X_{i,j}^{\Bn}(B_{\Bt}) = \{ 0,\dotsc, \Bt(i,j)-1\}$.  Then $B_{\Bt} \in \B_{\Bt}$, and write $\lambda_{\Bt}$ for $\Par(B_{\Bt})$.

\begin{prop}
Let $B \in \B_s$.  Then $B \in \B_{\Bt}$ if and only if $\Par(B)$ and $\lambda_{\Bt}$ have the same sets of nodes of $n$-residue $\sigma_j-s$ for all $j=0,\dotsc, r-1$.
\end{prop}

\begin{proof}
Note that $B_{\Bt}$ is the unique element in $\B_{\Bt}$ such that $|\Par(B_{\Bt})|$ is the least.  We prove by induction on $|\Par(B)|$.  If $|\Par(B)| \leq |\lambda_{\Bt}|$, then $B \in \B_{\Bt}$ if and only if $B = B_{\Bt}$.  On the other hand, $\lambda_{\Bt}$ is characterised by its sets of nodes of $n$-residue $\sigma_j-s$ for all $j=0,\dotsc, r-1$ and that all its removable nodes lie among these.  As such, any partition $\mu$ having the same sets of nodes of $n$-residue $\sigma_j-s$ as $\lambda_{\Bt}$ for all $j=0,\dotsc, r-1$ must have $[\mu] \supseteq [\lambda]$ as otherwise $\lambda$ would have a removable node not in $[\mu]$.  Thus, if $|\Par(B)| \leq |\lambda_{\Bt}|$, then $\Par(B)$ has the same sets of nodes of $n$-residue $\sigma_j-s$ as $\lambda_{\Bt}$ for all $j=0,\dotsc, r-1$ if and only if $\Par(B) = \lambda_{\Bt}$, or equivalently $B= B_{\Bt}$.  Thus the base case for the induction holds.

Suppose that $|\Par(B)| > |\lambda_{\Bt}|$.  If $B \in \B_{\Bt}$, then there exists $x \in B$, $x-1 \notin B$ and $x \not\equiv_n \sigma_j$ for all $j$.  Let $C = B \cup \{x-1\} \setminus \{ x\}$.  Then $C \in \B_{\Bt}$, and $\Par(C)$ is obtained from $\Par(B)$ by removing a node of $n$-residue $x-s$.  By induction, $\Par(C)$ has the same sets of nodes of $n$-residue $\sigma_j-s$ as $\lambda_{\Bt}$ for all $j=0,\dotsc, r-1$, and thus, so does $\Par(B)$.  On the other hand, if $\Par(B)$ has the same sets of nodes of $n$-residue $\sigma_j-s$ as $\lambda_{\Bt}$ for all $j=0,\dotsc, r-1$, then since $|\Par(B)| > |\lambda_{\Bt}|$, we see that $\Par(B)$ must have some removable node of $n$-residue not equal to $\sigma_j-s$ for all $j=0,\dotsc, r-1$.  Let $\mu$ be the partition obtained by removing this node.  Then $\t^{\Bn}_{\beta_s(\mu)} = \t^{\Bn}_{B}$, and $\mu$ has the same sets of nodes of $n$-residue $\sigma_j-s$ as $\lambda_{\Bt}$ for all $j=0,\dotsc, r-1$, so that by induction, $\Bt = \t^{\Bn}_{\beta_s(\mu)} = \t^{\Bn}_{B}$, and we are done.
\end{proof}

We get the following immediate corollary.

\begin{cor} \label{characterisation}
Let $\lambda,\mu \in \CP$ and $s \in \mathbb{Z}$.  Then $\t^{\Bn}_{\beta_s(\lambda)} = \t^{\Bn}_{\beta_s(\mu)}$ if and only if $\lambda$ and $\mu$ have the same sets of nodes of $n$-residue $\sigma_j -s$ for all $j = 0,\dotsc, r-1$.
\end{cor}

For each $\Bt \in \BT_{\Bn}$, let $\CF_{\Bt}$ be the $\mathbb{C}(q)$-vector space with basis $\B_{\Bt}$.  Then $\CF_{\Bt}$ is a vector subspace of $\CF_s$ if $\s(\Bt) = s$, and is invariant under the action of $\Up$.  Moreover, $\CF_s = \bigoplus_{\Bt \in \BT_{\Bn,s}} \CF_{\Bt}$.

In the special case where $\Bn = (n)$, we use the following notational conventions:
\begin{itemize}
\item $\BT_{(n)} = \{ \Bt : \BZ \to \{0,1,\dotsc,n\} \mid  \sum_{i\geq0} \Bt(i) +\sum_{i<0} (n - \Bt(i)) < \infty\}$.
\item If $\Bt \in \BT_{(n)}$, then $\s(\Bt) = \sum_{i\geq 0} \Bt(i) - \sum_{i<0} (n - \Bt(i))$.
\item If $B \subseteq \BZ$, then $X^{(n)}_i(B) = \{ x-in \mid x \in B,\ in \leq x < (i+1)n \}$ for all $i \in \BZ$, and $\t^{(n)}_{B} : \BZ \to \{0,1,\dotsc, n\}$ sends $i$ to $|X^{(n)}_i(B)|$.
\end{itemize}

We now proceed to identify $\CF_{\Bt}$ as a $\Up$-module. Let $V_m$ denote the natural module
for $U_q(\Gsl_m) = \left\langle e_1, \ldots, e_{m-1}; f_1, \ldots, f_{m-1};
K_1^{\pm},\ldots, K_{m-1}^{\pm}
\right\rangle$.  It has distinguished basis $v_{m,0}, \ldots v_{m,m-1}$, with action given by
$$
e_iv_{m,j} = \delta_{i,j} v_{m,j-1}, \qquad  f_i v_{m,j}  = \delta_{i-1,j} v_{m,j+1}$$
and
$$
K_i^+v_{m,j} =
\begin{cases}
qv_{m,j} & \text{if } j=i-1 \\
q^{-1}v_{m,j} & \text{if } j=i \\
v_{m,j} & \text{otherwise}.
\end{cases}
$$
The exterior power $\bw{d} V_m$ is an irreducible $U_q(\Gsl_m)$-module; its canonical basis is given by $\{v_{m,X} :=v_{m,i_1}\wedge\ldots \wedge v_{m,i_d} \mid X=\left\{i_1<\ldots<i_d\right\} \in \P_{m,d} \}$, where here and hereafter $\P_{m,d}$ denotes the set of $d$-element subsets of $\{0,1,\dotsc,m-1\}$. So
\begin{align*}
e_iv_{m,X}&=
\begin{cases}
v_{m,X\cup\{i-1\}\setminus\{i\}} & \text{if } i\in X \text{ and } i-1 \notin X \\
0 & \text{otherwise},
\end{cases}
\\
f_iv_{m,X}&=
\begin{cases}
v_{m,X\cup\{i\}\setminus\{i-1\}} & \text{if } i-1\in X \text{ and } i \notin X \\
0 & \text{otherwise},
\end{cases}
\\
K_i^+v_{m,X}&=
\begin{cases}
	qv_{m,X} & \text{if } i-1\in X \text{ and } i \notin X \\
	q^{-1} v_{m,X} & \text{if } i\in X \text{ and } i-1 \notin X \\
	v_{m,X} & \text{otherwise}.
\end{cases}
\end{align*}
It follows that for any tuple $\Bd = (d_1,\dotsc,d_r)$ of integers with $0 \leq d_j \leq n_j$ for all $j$, $\bw{\Bd}_{\:\Bn}(V) : = \bw{d_1} V_{n_1} \otimes \ldots
\otimes \bw{d_r} V_{n_r}$ is an irreducible $\Up$-module with
canonical basis $$C_{\Bd}=\{v_{n_1,X_1}\otimes\dotsb\otimes
v_{n_r,X_r} \mid X_j \in \P_{n_j,d_j}\ \forall j \}.$$

\begin{lemma} \label{isom}
Let $\Bt \in \BT_\Bn$.  Write $\bw{\Bt}(V)$ for $\bigotimes_{i \in \BZ,\, 1\leq j \leq r} \bw{\Bt(i,j)} V_{n_j}$.
\begin{enumerate}
\item The set $\left\{ \bigotimes_{i \in \BZ, 1 \leq j \leq r} v_{n_i,X_{i,j}} \mid X_{i,j} \in \P_{n_j,\Bt(i,j)}\ \forall i,j \right\}$ is a basis for $\bw{\Bt}(V)$.
\item The map $\B_{\Bt} \to \prod_{i \in \BZ, 1 \leq j \leq r} \P_{n_j,\Bt(i,j)}$ defined by $B \mapsto (X^{\Bn}_{i,j}(B))_{i \in \BZ, 1 \leq j \leq r}$ for all $B \in \B_{\Bt}$ is bijective.
\item The $\BC(q)$-linear map $\Theta_{\Bt} : \CF_{\Bt} \to \bw{\Bt}(V)$ defined by $B \mapsto \bigotimes_{i \in \BZ, 1 \leq j \leq r} v_{n_j,X^{\Bn}_{i,j}(B)}$ for all $B \in \B_{\Bt}$ is bijective, and is an isomorphism of $\Up$-modules if $\bw{\Bt}(V)$ is viewed as
    $$
    \dotsb \otimes \bw{\Bt_{-1}}_{\:\Bn}(V) \otimes \bw{\Bt_0}_{\:\Bn}(V) \otimes \bw{\Bt_1}_{\:\Bn}(V) \otimes \dotsb,
    $$ where $\Bt_i = (\Bt(i,1),\dotsc, \Bt(i,r))$ for all $i \in \BZ$.
\end{enumerate}
%For each $i = 1,\dotsc, r$, let $\Bt_i = (\Bt(i,1),\Bt(i,2),\dotsc,\Bt(i,r))$.  The $\BC(q)$-linear map $\theta_{\Bt} : \CF_{\Bt} \to \bigotimes_{i\in\mathbb{Z}} \bw{\Bt_i}_{\:\Bn} (V)$, defined by $B \mapsto \bigotimes_{i\in\mathbb{Z}} v_{n_1,X_{i,1}(B)} \otimes \dotsb \otimes v_{n_r,X_{i,r}(B)}$ for all $B \in \CF_{\Bt}$, is an isomorphism of $\Up$-modules.
\end{lemma}
Note that when $i \gg 0$ or $i \ll 0$, the component $\bw{\Bt(i,j)} V_{n_j}$, and hence $\bw{\Bt_i}_{\:\Bn}(V)$, is canonically isomorphic to the trivial module $\BC(q)$, so the tensor product $\bw{\Bt}(V)$ is essentially a finite one.

\begin{proof}
This is straightforward to check. The reader may find
\autoref{f:exampleabacusblobs} useful.
\end{proof}

Each of the irreducible $\Up$-modules
 $\bw{\Bd}_{\:\Bn}(V)$
with its canonical basis $C_{\Bd} = \{ v_{\BX} = v_{n_1,X_1} \otimes \dotsb \otimes v_{n_r,X_r} \mid \BX = (X_j)_{1 \leq j \leq r} \in \prod_{j=1}^r \P_{n_j,d_j} \}$ is a based module.
Hence, by Lusztig's theory, the tensor
product $\bw{\Bt}(V)  = \dotsb \otimes \bw{\Bt_{-1}}_{\:\Bn}(V) \otimes \bw{\Bt_0}_{\:\Bn}(V) \otimes \bw{\Bt_1}_{\:\Bn}(V) \otimes \dotsb$ is a based module with basis
$$\mathbf{C}^{\tp} =  \dotsb \tp C_{\Bt_{-1}} \tp C_{\Bt_0} \tp C_{\Bt_1} \tp \dotsb,$$
while $\mathbf{C}^{\otimes} = \dotsb \otimes C_{\Bt_{-1}} \otimes C_{\Bt_0} \otimes C_{\Bt_1} \otimes \dotsb$ is the basis given in \autoref{isom}(1).  For each $B \in \B_{\Bt}$, let $\BX_i(B) = (X^{\Bn}_{i,j}(B))_{1\leq j \leq r} \in \prod_{j=1}^{r} \P_{n_j,\Bt(i,j)}$ for each $i \in \BZ$, and let
\begin{align*}
v_B^{\otimes} &= \dotsb \otimes v_{\BX_{-1}(B)} \otimes v_{\BX_0(B)} \otimes v_{\BX_1(B)} \otimes \dotsb \in \mathbf{C}^{\otimes},\\
v_B^{\tp} &= \dotsb \tp v_{\BX_{-1}(B)} \tp v_{\BX_0(B)} \tp v_{\BX_1(B)} \tp \dotsb \in \mathbf{C}^{\tp}.
\end{align*}
Then $\mathbf{C}^{\otimes} = \{ v_B^{\otimes} \mid B \in \B_{\Bt} \} = \{ \Theta_{\Bt}(B) \mid B \in \B_{\Bt} \}$ by \autoref{isom}(3).  Also, $\mathbf{C}^{\tp} = \{v_B^{\tp} \mid B \in \B_{\Bt} \}$.

For each $B \in \B_{\Bt}$, let $G_{\Bt}(B) = \Theta_{\Bt}^{-1} (v_B^{\tp})$.  Denote by $\psi_{\Bt}$ the $\BC(q)$-semilinear involution of $\CF_{\Bt}$ fixing $G_{\Bt}(B)$ for each $B \in \B_{\Bt}$.

Let $s = \s(\Bt)$.  Recall that $\beta_s : \CP \to \B_s$ is a bijection, with inverse $\Par$.
Let $\CP_{\Bt} = \Par(\B_{\Bt})$, so that the following three statements are equivalent:
$$
(1)\ \mu \in \CP_{\Bt}; \qquad \qquad  (2)\ \beta_s(\mu) \in \B_{\Bt}; \qquad \qquad (3)\ \t_{\beta_s(\mu)} = \Bt.
$$
For each $\mu \in \CP_{\Bt}$, we write $G_{\Bt}(\mu)$ for $G_{\Bt}(\beta_{s}(\mu))$.
In addition define $d^{\,\Bt}_{\lambda\mu}(q) \in \BC(q)$ for $\lambda, \mu \in \CP_{\Bt}$ by
$$G_{\Bt}(\mu)=\sum_{\lambda \in\CP_{\Bt}} d^{\,\Bt}_{\lambda\mu}(q)\beta_{s}(\lambda).$$

From \autoref{l:canonicaltensorproduct}(1) we deduce the following.
\begin{lemma}\label{l:tensorbasistriangular}
Let $\lambda, \mu \in \CP_{\Bt}$.
\begin{enumerate}
\item $d^{\,\Bt}_{\mu \mu}(q) = 1$.
\item If $\lambda \ne \mu$, then $d^{\,\Bt}_{\lambda \mu}(q) \in q\BZ[q]$.
\item If $d^{\,\Bt}_{\lambda \mu}(q) \ne 0$, then $\mu \geq \lambda$.
\item If $m \in \CF_{\Bt}$ satisfies $\psi_{\Bt}(m) = m$ and $m - \beta_s(\mu) \in \sum_{B \in \B_{\Bt}} q\BZ[q] B$, then $m = G_{\Bt}(\mu)$.
%$|\lambda|=|\mu|$ and
\end{enumerate}
\end{lemma}

\begin{prop} \label{tpcan}
Let $\Bt \in \BT_{\Bn}$.  For each $j = 1,\dotsc,r$, define $\Bt_j : \BZ \to \BN_0$ by $\Bt_j(i) = \Bt(i,j)$.
\begin{enumerate}
\item For each $j = 1,\dotsc, r$, $\Bt_j \in \BT_{(n_j)}$.  In addition, $\sum_{j=1}^r \s(\Bt_j) = \s(\Bt)$.
\item The $\BC(q)$-linar map $\Phi_{\Bt}: \CF_{\Bt} \to \CF_{\Bt_1} \otimes \dotsb \otimes \CF_{\Bt_r}$ defined by $\Phi_{\Bt}(B) = B_1 \otimes \dotsb \otimes B_r$ for all $B \in \B_{\Bt}$, where $B_j = \{ x + in_j \mid i \in \BZ,\ x\in X^{\Bn}_{i,j}(B) \}$ for all $j$, is a well-defined isomorphism of $\Up$-modules (where we identify $\Up$ with $U_q(\mathfrak{sl}_{n_1}) \otimes \dotsb \otimes U_q(\mathfrak{sl}_{n_r})$ in the obvious way), and
    $$
    \Phi_{\Bt}(G_{\Bt}(B)) = G_{\Bt_1}(B_1) \otimes \dotsb \otimes G_{\Bt_r}(B_r)$$
    for all $B \in \B_{\Bt}$.
\end{enumerate}
\end{prop}

\begin{proof}
We show only the last equality; all other assertions can be easily verified.

By \autoref{isom}(2), we have, for each $j = 1,\dotsc, r$, a bijective $\BC(q)$-linear map (actually an isomorphism of $U_q(\mathfrak{sl}_{n_j})$-modules) $\Theta_{\Bt_j} : \CF_{\Bt_j} \to \bigotimes_{i \in \BZ} \bw{\Bt_j(i)} V_{n_j}$ sending $B_j$ to $\bigotimes_{i \in \BZ} v_{n_j, X^{(n_j)}_i(B_j)}$ for all $B_j \in \B_{\Bt_j}$.  Since
{\allowdisplaybreaks
\begin{align*}
\left((\Theta_{\Bt_1}^{-1} \otimes \dotsb \otimes \Theta_{\Bt_r}^{-1}) \circ \Theta_{\Bt}\right)(B) &=
(\Theta_{\Bt_1}^{-1} \otimes \dotsb \otimes \Theta_{\Bt_r}^{-1}) \left( \bigotimes_{\substack{i\in \BZ \\ 1\leq j \leq r}} v_{n_j, X^{\Bn}_{i,j}(B)} \right) \\
&= \bigotimes_{j=1}^r \Theta_{\Bt_j}^{-1} \left(\bigotimes_{i \in \BZ} v_{n_j, X^{(n_j)}_i(B_j)}\right) = \bigotimes_{j=1}^r B_j = \Phi_{\Bt}(B)
\end{align*}
for all $B \in \B_{\Bt}$, we see that $\Phi_{\Bt} = (\Theta_{\Bt_1}^{-1} \otimes \dotsb \otimes \Theta_{\Bt_r}^{-1}) \circ \Theta_{\Bt}$.  Thus
\begin{align*}
\Phi_{\Bt}(G_{\Bt}(B)) &= (\Theta_{\Bt_1}^{-1} \otimes \dotsb \otimes \Theta_{\Bt_r}^{-1})(\Theta_{\Bt} (G_{\Bt}(B))) \\
&= (\Theta_{\Bt_1}^{-1} \otimes \dotsb \otimes \Theta_{\Bt_r}^{-1})(v_B^{\tp}) \\
&= (\Theta_{\Bt_1}^{-1} \otimes \dotsb \otimes \Theta_{\Bt_r}^{-1})\left(\bigtp_{i \in \BZ} \left(\bigotimes_{j=1}^r v_{n_j, X^{\Bn}_{i,j}(B)}\right)\right) \\
%&= (\Theta_{\Bt_1}^{-1} \otimes \dotsb \otimes \Theta_{\Bt_r}^{-1})(\bigotimes_{j=1}^r v_{n_j, X_{i,j}(B) \\
&= (\Theta_{\Bt_1}^{-1} \otimes \dotsb \otimes \Theta_{\Bt_r}^{-1}) \left( \bigotimes_{j=1}^r \left(\bigtp_{i \in \BZ} v_{n_j, X^{(n_j)}_i(B_j)} \right) \right) \\
&= \bigotimes_{j=1}^r \Theta_{\Bt_j}^{-1} \left(\bigtp_{i \in \BZ} v_{n_j, X^{(n_j)}_i(B_j)}\right) \\
&= \bigotimes_{j=1}^r G_{\Bt_j}(B_j),
\end{align*}
}
using \autoref{P:internal-external} for the fourth equality.
\end{proof}

\section{Comparison of canonical bases} \label{compare}

Fix $n \in \BN$ and a tuple $\Bn = (n_1,\dotsc,n_r)$ of positive integers such that $n_1+ \dotsb + n_r = n$.  Recall our subalgebra $\Up$ of $\U$ and our set $\BT_\Bn$ defined in the last section.  We define a partial order $\geq$ on $\BT_\Bn$ as follows:  for $\Bt, \Bt' \in \BT_\Bn$, we have $\Bt > \Bt'$ if and only if $\s(\Bt) = \s(\Bt')$ and there exists $(i_0,j_0) \in \mathbb{Z} \times \{1,\dotsc, r\}$ such that $\Bt(i_0,j_0) > \Bt'(i_0,j_0)$ and %$\Bt(i_0,j) = \Bt'(i_0,j)$ for all $j > j_0$, and $\Bt(i,j) = \Bt'(i,j)$ for all $i > i_0$ and all $j$.
$\Bt(i,j) = \Bt'(i,j)$ for all $(i,j)$ with $i>i_0$ or both $i= i_0$ and $j>j_0$.
Note that $\geq$ restricts to a total order on $\BT_{\Bn,s}$ for each $s \in \BZ$.

\begin{lemma} \label{L:order}
Let $\lambda,\mu \in \CP$.  If $\lambda \geq_{J_n} \mu$, then $\t^{\Bn}_{\beta_s(\lambda)} \geq \t^{\Bn}_{\beta_s(\mu)}$ for all $s \in \BZ$.

%In particular, $G(\mu) \in \bigoplus_{\Bt \leq \t_{\mu,s}} \CF_{\Bt}$.
\end{lemma}

\begin{proof}
Note that $B \to_n C$ implies $\t^{\Bn}_{B} \geq \t^{\Bn}_{C}$.
\end{proof}

For $\Bt \in \BT_{\Bn,s}$ let $\pi_{\Bt}:\CF_{s} \to \CF_\Bt$ be the natural projection defined by $\pi_{\Bt}(B) = B$ if $B \in \B_{\Bt}$ and zero otherwise.

Since, for each $s \in \BZ$, $\CP = \bigcup_{\Bt \in \BT_{\Bn,s}} \CP_{\Bt}$ is a partition of $\CP$, and the order $\geq$ on $\BT_{\Bn,s}$ defined above respects the Jantzen order $\geq_{J_n}$ by Lemma \ref{L:order}, we can apply Lemma \ref{genind} with $H(\lambda) = G(\lambda)$ and $\xi$ an element of $\Up$, or $H(\lambda) = [P(\lambda)]$ and $\xi$ an $i\text{-}\mathrm{Ind}$ functor.

\begin{thm} \label{T:decomp}
Let $s \in \BZ$ and $\Bt \in \BT_{\Bn,s}$.
If $\mu \in \CP_{\Bt}$, then $\pi_{\Bt} (G_s(\mu))=G_{\Bt}(\mu)$.
Equi\-valently $d_{\lambda\mu}(q)= d^{\,\Bt}_{\lambda\mu}(q)$ for all $\lambda,\mu\in\CP_{\Bt}$.
\end{thm}

\begin{proof}
We prove by induction on $|\mu|$ the following equivalent statement:  if $\mu \in \CP$, then $\pi_{\t^{\Bn}_{\beta_s(\mu)}} (G_s(\mu))=G_{\t^{\Bn}_{\beta_s(\mu)}}(\mu)$ for all $s \in \mathbb{Z}$.
If $\mu=\emptyset$, then $G(\mu) = \mu$ by \autoref{T:appendix}(1,3), while $G_{\t^{\Bn}_{\beta_s(\mu)}}(\mu) = \beta_s(\mu)$ by \autoref{l:tensorbasistriangular}(1,3), so that the result holds.  So let $\mu\neq\emptyset$ and assume that the statement holds for all partitions $\tilde{\mu}$ such that $|\tilde{\mu}| < |\mu|$, or both $|\tilde{\mu}| = |\mu|$ and $\tilde{\mu} < \mu$.  Let $s \in \BZ$ and let $\Bt = \t^{\Bn}_{\beta_s(\mu)}$. Let $c=\max\{i \in \BZ \mid \Bt(i,j)>0 \text{ for some } j\}$, so that row $c$ is the bottommost row that contains some bead in the $n$-abacus display of $\beta_s(\mu)$. %($= \lfloor \beta_s(\mu)_1/n \rfloor$).
There are two cases to consider.

If $X^{\Bn}_{c,j}(\beta_s(\mu))=\{0,\ldots, \Bt(c,j)-1\}$ for all $j$ (i.e.\ the beads occurring in each section of row $c$ in the $n$-abacus display of $\beta_s(\mu)$ occupy the leftmost positions), then
$v_{n_1,X^{\Bn}_{c,1}(\beta_s(\mu))}\otimes \cdots \otimes v_{n_r,X^{\Bn}_{c,r}(\beta_s(\mu))}$ is the highest weight vector
for the irreducible $\Up$-module
$\bw{\Bt(c,1)}(V_{n_1})\otimes \cdots \otimes \bw{\Bt(c,r)}(V_{n_r})$.

Let $\lambda \in \CP_{\Bt}$.  Note that if $\lambda>\mu$, then
$d_{\lambda\mu}(q) = 0$, by \autoref{T:appendix}(3), and
$d^{\,\Bt}_{\lambda\mu}(q)=0$ by \autoref{l:tensorbasistriangular}(3). So we may assume that $\lambda \leq \mu$.  In this case, the beads in each section of row $c$ of the $n$-abacus display of $\beta_s(\lambda)$ must be occupying the leftmost positions too, i.e. $X^{\Bn}_{c,j}(\beta_s(\lambda))=X^{\Bn}_{c,j}(\beta_s(\mu))$ for all $j$.
Let $\Bt' \in \BT_{\Bn}$ such that
$$
\Bt'(i,j) =
\begin{cases}
\Bt(i,j), &\text{if } i \ne c; \\
0, &\text{if } i =c.
\end{cases}
$$
Write $s' = \s(\Bt')$. % ($=s - \sum_{j=1}^r \Bt(c,j)$).
Let $\tilde{\lambda}, \tilde{\mu} \in \CP$ such that $\beta_{s'}(\tilde{\lambda})$ and $\beta_{s'}(\tilde{\mu})$ are obtained by removing the beads in row $c$ from the $n$-abacus displays of $\beta_s(\lambda)$ and $\beta_s(\mu)$ respectively.  Then $\tilde{\lambda}, \tilde{\mu} \in \CP_{\Bt'}$ and
% \in \CP_{\Bt'}$ be such that $X_{i,j}(\beta_s'(\tilde{\lambda}) = X_{i,j}(\beta_s(\lambda))$, $X_{i,j}(\beta_{s'}(\tilde{\mu})) = X_{i,j}(\beta_s(\mu))$ for all $i \ne c$ and for all $j$, and $X^{s'}_{c,j}(\tilde{\lambda}) = X^{s'}_{c,j}(\tilde{\mu}) = \emptyset$ for all $j$.  Thus $\tilde{\lambda}$ and $\tilde{\mu}$
are the partitions obtained from $\lambda$ and $\mu$ respectively by removing the first $\sum_{j=1}^r \Bt(c,j)$ parts, so that
$\left|\tilde{\mu}\right|<\left|\mu\right|$. We have
$$d_{\lambda\mu}(q)=d_{\tilde{\lambda}\tilde{\mu}}(q)=
d^{\,\Bt'}_{\tilde{\lambda}\tilde{\mu}}(q)= d^{\,\Bt}_{\lambda\mu}(q),$$
where the first equality is `row removal' \cite[Theorem 1(1)]{CMT}, the second by induction and the last by
\autoref{l:canonicaltensorproduct}(2).

We now proceed to the second case.  In this case, there is a bead in row $c$ of the $n$-abacus display of $\beta_s(\mu)$ having a vacant preceding position that is in the same section as the bead.  Formally, there exists $k \in \beta_{s}(\mu)$ such that $\lfloor \frac{k}{n} \rfloor = c$, $k-1 \notin \beta_{s} (\mu)$, and $f_k \in \Up$.
Let $\lambda= \Par(\beta_{s}(\mu) \cup \{k-1\} \setminus \{k\})$.
Then $\lambda \in \CP_{\Bt}$, and we have, by \autoref{inductivecor}(2),
\begin{comment}
  Let $c'$ be the largest integer $j$ such that $X^s_{c,j}(\mu)\ne\{0,\ldots, \Bt(c,j)-1\}$, and let $k$ be the largest integer such that $k \in X^s_{c,c'}(\mu)$ and $k-1 \notin X^s_{c,c'}(\mu)$.  Let $\lambda$ be the partition such that $X^s_{i,j}(\lambda) = X^s_{c,c'}(\mu) \cup\{k-1\} \setminus \{k\}$ if $(i,j) = (c,c')$ and $X^s_{i,j}(\lambda) = X^s_{i,j}(\mu)$ otherwise.  Then $|\lambda| = |\mu| -1$, and
\end{comment}
$$f_{k}(G_s(\lambda)) = G_s(\mu) + \sum_{\substack{\nu \in \CP \\ |\nu| = |\mu| \\ \nu < \mu}} L_\lambda^{\nu}(q) G_s(\nu),$$ %where $k'=n_1+\ldots+n_{c'-1}+k$
with $L_\lambda^{\nu}(q)=L_\lambda^{\nu}(q^{-1})$ for all $\nu$ by \autoref{T:appendixlast}.  Applying \autoref{genind} with $\xi = f_k$, we get
$$
f_k(\pi_{\Bt}(G_s(\lambda))) = \pi_{\Bt}(f_k(G_s(\lambda))) = \pi_{\Bt}(G_s(\mu)) + \sum_{\substack{\nu \in \CP_{\Bt} \\ |\nu| = |\mu| \\ \nu < \mu }} L_\lambda^{\nu}(q) \pi_{\Bt}(G_s(\nu)).$$
Applying the induction hypothesis to $\lambda$ and $\nu$, we get
\begin{align*}
\pi_{\Bt} (G_s(\mu)) &=
f_{k}(G_{\Bt}(\lambda)) -
\sum_{\substack{\nu \in \CP_{\Bt} \\|\nu| = |\mu| \\ \nu < \mu}} L_\lambda^{\nu}(q) G_{\Bt}(\nu).
\end{align*}
Thus, $\pi_{\Bt}(G_s(\mu))$
is $\psi_{\Bt}$-invariant.  Furthermore, since $G(\mu) - \mu \in \bigoplus_{\nu \in \CP} q\mathbb{Z}[q]\nu$, we see that
$$G_s(\mu) - \beta_s(\mu) = \beta_s(G(\mu) - \mu) \in \bigoplus_{\nu \in \CP} q\mathbb{Z}[q]\beta_s(\nu) = \bigoplus_{B \in \B_s} q\mathbb{Z}[q]B,$$
so that
$$\pi_{\Bt}(G_s(\mu)) - \beta_s(\mu) = \pi_{\Bt}(G_s(\mu) - \beta_s(\mu)) \in \bigoplus_{B \in \B_{\Bt}} q\mathbb{Z}[q]B.$$  Hence
$\pi_{\Bt} (G_s(\mu))= G_{\Bt}(\mu)$ by \autoref{l:tensorbasistriangular}(4), as desired.
\end{proof}

%We get a strengthening of \autoref{l:tensorbasistriangular}(3).

\begin{cor}
Let $\Bt \in \BT_{\Bn}$, and let $\lambda,\mu \in \CP_{\Bt}$.  Then $d^{\,\Bt}_{\lambda\mu}(q) \ne 0$ only if $\mu \geq_{J_n} \lambda$.
\end{cor}

\begin{proof}
This follows immediately from \autoref{T:decomp} and \autoref{T:appendix}(3) .
\end{proof}

\begin{cor} \label{indproj}
Let $s \in \BZ$ and $\Bt \in \BT_{\Bn,s}$.  Let $\lambda \in \CP_{\Bt}$, and let $\xi \in \Up$.  Suppose that $\xi(G_s(\lambda)) = \sum_{\mu \in \CP} a_\mu(q)G_s(\mu)$.  Then
$$\xi(G_{\Bt}(\lambda)) = \sum_{\mu \in \CP_{\Bt}} a_\mu(q) G_{\Bt}(\mu).$$

%Let $\lambda \in \CP_{\Bt}$, and let $f$ be a Chevalley generator of $\Up^-$.  Suppose that $f(G(\lambda)) = \sum_{\mu \in \CP} L_{\mu}(q)G(\mu)$.  Then $f(G_{\Bt}(\lambda)) = \sum_{\mu \in \CP_{\Bt}} L_{\mu}(q) G_{\Bt}(\mu)$.
\end{cor}

\begin{proof}
This follows immediately from \autoref{genind} and \autoref{T:decomp}.
\end{proof}

The next result may be regarded as a `runner removal' theorem.

\begin{theorem} \label{T:product}
Let $s \in \BZ$ and $\Bt \in \BT_{\Bn,s}$.  Let $\lambda,\mu \in \CP_{\Bt}$.  For each $j = 1,\dotsc, r$, let $s_j = \sum_{i \geq 0} \Bt(i,j) +\sum_{i<0} (\Bt(i,j) - n_j)$, and let $\lambda^{(j)}, \mu^{(j)} \in \CP$ such that \begin{align*}
\beta_{s_j}(\lambda^{(j)}) &= \{x+ in_j \mid i \in \BZ,\ x \in X^{\Bn}_{i,j}(\beta_s(\lambda)) \},\\
\beta_{s_j}(\mu^{(j)}) &= \{x+ in_j \mid i \in \BZ,\ x \in X^{\Bn}_{i,j}(\beta_s(\mu))\}.
\end{align*}
(Thus, $\lambda^{(j)}$ is the partition read off from the $n_j$-abacus consisting of runners $i$ for $\sigma_{j-1} \leq i < \sigma_j$ in the $n$-abacus display of $\beta_s(\lambda)$.  Similarly for $\mu^{(j)}$.)
Then
$$d_{\lambda\mu}^{\,n}(q) = \prod_{j=1}^r d_{\lambda^{(j)}\mu^{(j)}}^{\,n_j}(q).$$
\end{theorem}

\begin{proof}
For each $j=1,\dotsc,r$, define $\Bt_j : \BZ \to \BN_0$ by $\Bt_j(i) = \Bt(i,j)$ for all $i \in \BZ$.
By \autoref{tpcan}, we have
\begin{align*}
\Phi_{\Bt}(\beta_s(\mu)) = \bigotimes_{j=1}^r \beta_{s_j}(\mu^{(j)}) \qquad \text{and} \qquad \Phi_{\Bt}(G_{\Bt}(\mu)) = \bigotimes_{j=1}^r G_{\Bt_j}(\mu^{(j)})
\end{align*}
for all $\mu \in \CP_{\Bt}$.  Thus, $d^{\,\Bt}_{\lambda\mu}(q) = \prod_{j=1}^r d^{\,\Bt_j}_{\lambda^{(j)}\mu^{(j)}}(q)$.  Now apply \autoref{T:decomp}.
\end{proof}

We end this section with a discussion of the case where $\Bn = (n_1,\dotsc,n_r)$ has $n_j = 2$ for some $j$.  We first look at the special case where $\Bn = (2)$.  In this case, since $\bw{1} V_2 = V_2$ and $\bw{0} V_2 = \bw{2} V_2 = \BC$, for each $\Bt \in \BT_{\Bn}$, $\CF_{\Bt}$ is isomorphic to a tensor power $V_2^{\otimes d}$ as $U_q(\Gsl_2)$-modules.  The latter has been studied extensively by Frenkel and Khovanov in \cite{FK}.  In particular, they provide closed formulas for $v_B^{\tp}$ in terms of graphical calculus combinatorics for each $B \in \B_{\Bt}$, so that there are closed formulas for $d^{\,\Bt}_{\lambda\mu}(q)$ for $\lambda,\mu \in \CP_{\Bt}$.  They also describe $f(v_B^{\tp})$, which we translate to our language as follows:

\begin{theorem}[{\cite{FK}}] \label{FKthm}
Let $\Bn = (2)$, and let $\lambda \in \CP$.  Let $s\in \BZ$ and write $\Bt = \t^{\Bn}_{\beta_s(\lambda)}$.  Suppose that
$$
f(G_{\Bt}(\lambda)) = \sum_{\mu \in \CP_{\Bt}} L_\lambda^{\mu}(q) G_{\Bt}(\mu).
$$
Then $$
L_{\lambda}^{\mu}(q) =
\begin{cases}
[1 + n_{\lambda\mu}]_q, &\text{if $\lambda$ can be obtained from $\mu$} \\
&\raisebox{5pt}{\text{\qquad by removing a normal node $\mathfrak{n}_{\lambda\mu}$ of $2$-residue $s+1$;}} \\
0, &\text{otherwise}.
\end{cases}
$$
Here $n_{\lambda\mu}$ is the number of normal nodes of $\mu$ with $2$-residue $s+1$ and to the right of $\mathfrak{n}_{\lambda\mu}$.
\end{theorem}

\begin{proof}
We briefly describe our translation of the results in \cite{FK}.  The basis vector $v_{2,1}$ of $V_2$, represented by an up arrow in the diagrams in \cite{FK}, corresponds to a removable node of $2$-residue $s+1$, while the other basis vector $v_{2,0}$, represented by a down arrow in the diagrams in \cite{FK}, corresponds to an addable node of $2$-residue $s+1$.  An arc connecting the up and down arrows induces a pairing of the corresponding removable node with the addable node.
% On the other hand, Kleshchev \cite[Theorem 1.4]{Kl}
%  proves that some of the integers $L_\nu$ may be
%  described by the combinatorics of certain sign sequences.
%  The signs $+$ and $-$ correspond to the vectors
%  $v_{2,1}$ and $v_{2,0}$ of $V_2$, which in turn are represented by up and down arrows in the diagrams of \cite{FK}.
%  Kleshchev's pairing of $+$ and $-$ signs corresponds
%  to the connecting of up and down arrows by an arc. A full
%comparison of the descriptions of
%Kleshchev and of Frenkel and Khovanov reveals that
%  $L_{\nu}=L_{\nu}(1)$ whenever $\nu\in\CP_{\Bt}$ and $L_\nu(1)\neq 0$;
%hence
%$L_{\nu} \geq L_{\nu}(1)$ for all $\nu\in\CP_{\Bt}$.
We should point out that the description in \cite[p.445-7]{FK} is given in terms of the basis dual to $\{G_{\Bt}(\lambda)\}$ with respect to an inner product for which $e$ and $f$ are adjoint operators, and so we use their description of the action of $e$ rather than that of $f$.
\end{proof}

\begin{cor} \label{LBt}
Let $s \in \BZ$, and $\lambda \in \CP$.  Write $\Bt$ for $\t^{\Bn}_{\beta_s(\lambda)}$.  Suppose that $\Bn = (n_1,\dotsc, n_r)$ has $n_d = 2$ for some $d$, and let $k = \sum_{j=1}^{d-1} n_j+1$.  Let
$$
f_k(G_{\Bt}(\lambda)) = \sum_{\mu \in \CP_{\Bt}} L_\lambda^{\mu}(q) G_{\Bt}(\mu).
$$
Then $$
L_{\lambda}^{\mu}(q) =
\begin{cases}
[1 + n_{\lambda\mu}]_q, &\text{if $\lambda$ can be obtained from $\mu$} \\
&\raisebox{5pt}{\text{\qquad by removing a normal node $\mathfrak{n}_{\lambda\mu}$ of $n$-residue $k-s$;}} \\
0, &\text{otherwise}.
\end{cases}
$$
Here $n_{\lambda\mu}$ is the number of normal nodes of $\mu$ with $n$-residue $k-s$ and to the right of $\mathfrak{n}_{\lambda\mu}$.
\end{cor}

\begin{proof}
By \autoref{tpcan} and adopting the notations there with $\lambda = \Par(B)$ and $\lambda^{(i)} = \Par(B_i)$, we have
\begin{align*}
f_k(G_{\Bt}(\lambda))) &= \Phi_{\Bt}^{-1} ( G_{\Bt_1}(\lambda^{(1)}) \otimes \dotsb \otimes G_{\Bt_{d-1}}(\lambda^{(d-1)}) \otimes  \\
& \qquad\qquad f(G_{\Bt_d}(\lambda^{(d)})) \otimes G_{\Bt_{d+1}}(\lambda^{(d+1)}) \otimes \dotsb \otimes G_{\Bt_r}(\lambda^{(r)})) \\[6pt]
&= \Phi_{\Bt}^{-1} ( G_{\Bt_1}(\lambda^{(1)}) \otimes \dotsb \otimes G_{\Bt_{d-1}}(\lambda^{(d-1)}) \otimes \\
&\qquad  \sum_{\mu^{(d)} \in \CP_{\Bt_d}} L_{\lambda^{(d)}}^{\mu^{(d)}}(q) (G_{\Bt_d}(\mu^{(d)})) \otimes G_{\Bt_{d+1}}(\lambda^{(d+1)}) \otimes \dotsb \otimes G_{\Bt_r}(\lambda^{(r)}))
\end{align*}
The result thus follows from \autoref{FKthm}.
\end{proof}

\begin{cor}
Let $\lambda, \mu \in \CP$.  Suppose that there exists a residue class $i$ modulo $n$ such that $\lambda$ and $\mu$ has the same sets of nodes of $n$-residue $i-1$ and $i+1$.
%\begin{itemize}
%\item $|[\mu]_i| = |[\lambda]_i| +1$,
%\item $|[\mu]_a| = |[\lambda]_a|$ for all $a \ne i$,
%\item $[\mu]_{i\pm 1} = [\lambda]_{i\pm 1}$ (where $i-1$ is read as $n-1$ if $i = 0$ and $i+1$ is read as $0$ if $i =n-1$).
%\end{itemize}
Let $s \in \BZ$.  Writing $f_{s+i}(G_s(\lambda)) = \sum_{\nu \in \CP} L_s^{\nu}(q) G_s(\nu)$, we have
$$
L_s^{\mu}(q) =
\begin{cases}
[1 + n_{\lambda\mu}]_q, &\text{if $\lambda$ can be obtained from $\mu$} \\
&\raisebox{5pt}{\text{\qquad by removing a normal node $\mathfrak{n}_{\lambda\mu}$ of $n$-residue $i$;}} \\
0, &\text{otherwise}.
\end{cases}
$$
Here, $n_{\lambda\mu}$ is the number of normal nodes of $\mu$ with $n$-residue $i$ and to the right of $\mathfrak{n}_{\lambda\mu}$.
\end{cor}

\begin{proof}
Let $\Bn = (2,n-2)$, and let $x = 1-i$.  By \autoref{characterisation}, our conditions on $\lambda$ and $\mu$ force $\t^{\Bn}_{\beta_x(\lambda)} = \t^{\Bn}_{\beta_x(\mu)}$.  Thus, by \autoref{indproj} and \autoref{LBt} (with $k=1$), we see that the result holds for $L_x^{\nu}(x)$.  To deal with a general $s$, let $\theta_{ab} : \CF_a \to \CF_b$ ($a,b \in \BZ$) be the $\BC(q)$-linear isomorphism defined by $\theta_{ab} (\beta_a(\nu)) = \beta_b(\nu)$ for all $\nu \in \CP$.  Then $\theta_{ab}(G_a(\lambda)) = G_b(\nu)$ for all $\nu \in \CP$.  Furthermore $\theta_{ab}$ intertwines $f_r$ with $f_{r+b-a}$ for all residue class $r$ modulo $n$.  Thus, applying $\theta_{xs}$ to $f_1(G_x(\lambda)) = \sum_{\nu \in \CP} L_x^{\nu}(q) G_x(\nu)$, we get $f_{s+i}(G_s(\lambda)) = \sum_{\nu \in \CP} L_x^{\nu}(q) G_s(\nu)$, and the result follows.
\end{proof}

\section{Decomposition numbers} \label{S:decompnumbers}

Recall that we identified the Grothendieck group $K_0(\bigoplus_m \mathbf{S}_m\text{-\textbf{mod}})$ with the classical (non-quantized) Fock space $\CF_c$.  Recall also the $A$-lattice $\CF_{\BZ,A}$ (where $A = \BZ[q,q^{-1}]$) and the $\BZ$-linear map $\E : \CF_{\BZ,A} \to \CF_c$ which intertwines the Chevalley generators of $\U_A^-$ with the $i$-Ind functors.

We assume in this section that $n= p$, the characteristic over which the Schur algebras $\BS_m$ are defined.  Let $\Bn = (n_1,\dotsc, n_r)$ be a tuple of positive integers summing to $p$.  For each $\Bt \in \BT_{\Bn}$, define $\pie_{\Bt} : \CF_c \to \CF_c$ to be the natural projection such that $\pie_{\Bt}(\lambda) = \lambda$ if $\lambda \in \CP_{\Bt}$ and zero otherwise, and we write $\CF_{c,\Bt}$ for the image of $\pie_{\Bt}$.  In addition, if $M$ is a module of a Schur algebra, we write $[M]_{\Bt}$ for $\pie_{\Bt}[M]$.

\begin{thm} \label{T:Schur}
Let $\Bn = (n_1, \dotsc, n_r)$ be a tuple of positive integers summing to $p$.
Suppose that $n_j \in \{1,2\}$ for all $j$, so that $\Up$ is isomorphic to
a tensor product of copies of $U_q(\Gsl_2)$.  Let $\Bt \in \BT_{\Bn}$. Then for all $\mu\in\CP_{\Bt}$, we have
$[P(\mu)]_{\Bt} = \E(G_{\Bt}(\mu))$. Equivalently,
$d_{\lambda\mu} = d^{\,\Bt}_{\lambda\mu}(1)$ for all $\lambda,\mu \in \CP_{\Bt}$.
\end{thm}

\begin{proof}
We follow an inductive argument reminiscent of the proof of \autoref{T:decomp}.  Once again, we prove the following equivalent statement: if $\mu \in \CP$, then $[P(\mu)]_{\t^{\Bn}_{\beta_s(\mu)}} = \E(G_{\t^{\Bn}_{\beta_s(\mu)}}(\mu))$ for all $s \in \BZ$.
If $\mu=\emptyset$ the result is clear. So let $\mu\neq\emptyset$ and assume the statement holds for all partitions $\tilde{\mu}$ such that $|\tilde{\mu}| < |\mu|$, or both $|\tilde{\mu}| = |\mu|$ and $\tilde{\mu} < \mu$.
Let $s \in \BZ$ and let $\Bt = \t^{\Bn}_{\beta_s(\mu)}$. Let $c=\max\{i \in \BZ \mid \Bt(i,j)>0 \text{ for some } j\}$.
There are two cases to consider.

If $X^{\Bn}_{c,j}(\beta_s(\mu))=\{0,\ldots, \Bt(c,j)-1\}$ for all $j$ (i.e.\ the beads occurring in each section of row $b$ in the $n$-abacus display of $B$ occupy the leftmost positions), then
$v_{n_1,X^{\Bn}_{c,1}(\beta_s(\mu))}\otimes \cdots \otimes v_{n_r,X^{\Bn}_{c,r}(\beta_s(\mu))}$ is the highest weight vector
for the irreducible $\Up$-module
$\bw{\Bt(c,1)}(V_{n_1})\otimes \cdots \otimes \bw{\Bt(c,r)}(V_{n_r})$.

Let $\lambda \in \CP_{\Bt}$.  Note that if $\lambda>\mu$, then
$d_{\lambda\mu} = 0$ by \autoref{decompnumbers}(2), and
$d^{\,\Bt}_{\lambda\mu}(q)=0$ by \autoref{l:tensorbasistriangular}(3). So we may assume that $\lambda \leq \mu$.  In this case,
$X^{\Bn}_{c,j}(\beta_s(\lambda))=X^{\Bn}_{c,j}(\beta_s(\mu))$ for all $j$.
Let $\Bt' \in \BT_{\Bn}$ such that
$$
\Bt'(i,j) =
\begin{cases}
\Bt(i,j), &\text{if } i \ne c; \\
0, &\text{if } i =c.
\end{cases}
$$
Write $s' = \s(\Bt')$. % ($=s - \sum_{j=1}^r \Bt(c,j)$).
Let $\tilde{\lambda}, \tilde{\mu} \in \CP$ such that $\beta_{s'}(\tilde{\lambda})$ and $\beta_{s'}(\tilde{\mu})$ may be obtained by removing the beads in row $c$ from the $n$-abacus displays of $\beta_s(\lambda)$ and $\beta_s(\mu)$ respectively.  Then $\tilde{\lambda}, \tilde{\mu} \in \CP_{\Bt'}$ and
% \in \CP_{\Bt'}$ be such that $X_{i,j}(\beta_s'(\tilde{\lambda}) = X_{i,j}(\beta_s(\lambda))$, $X_{i,j}(\beta_{s'}(\tilde{\mu})) = X_{i,j}(\beta_s(\mu))$ for all $i \ne c$ and for all $j$, and $X^{s'}_{c,j}(\tilde{\lambda}) = X^{s'}_{c,j}(\tilde{\mu}) = \emptyset$ for all $j$.  Thus $\tilde{\lambda}$ and $\tilde{\mu}$
are the partitions obtained from $\lambda$ and $\mu$ respectively by removing the first $s-s'$ parts, so that
$\left|\tilde{\mu}\right|<\left|\mu\right|$. We have
$$d_{\lambda\mu}=d_{\tilde{\lambda}\tilde{\mu}}=
d^{\,\Bt'}_{\tilde{\lambda}\tilde{\mu}}(1)= d^{\,\Bt}_{\lambda\mu}(1),$$
where the first equality is `row removal' \cite[Theorem 6.18]{Jam}, the second by induction and the last by
\autoref{l:canonicaltensorproduct}(2).

We now proceed to the second case.  In this case, there is a bead in row $c$ of the $n$-abacus display of $\beta_s(\mu)$ having a vacant preceding position that is in the same section as the bead.  Formally,
there exists $k \in \beta_{s}(\mu)$ such that $\lfloor \frac{k}{n} \rfloor = c$, $k-1 \notin \beta_{s} (\mu)$, and $f_k \in \Up$.
Let $\lambda$ be the partition such that $\beta_{s}(\lambda) = \beta_{s}(\mu) \cup \{k-1\} \setminus \{k\} $.
Then $\lambda \in \CP_{\Bt}$, and we have, by \autoref{inductivecor} and \autoref{indproj},
\begin{align*}
f_{k}(G_{\Bt}(\lambda)) &= G_{\Bt}(\mu) + \sum_{\substack{\nu\in\CP_{\Bt}\\ |\nu| = |\mu| \\ \nu < \mu}} L_\lambda^{\nu}(q) G_{\Bt}(\nu) \\
[(k-s)\text{-Ind} (P(\lambda))]&= [P(\mu)] + \sum_{\substack{\nu\in\CP \\ |\nu| = |\mu| \\ \nu < \mu}} L_{\lambda}^{\nu} [P(\nu)],
\end{align*}
and $L_\lambda^\nu = [(k-s)\text{-Res} (L(\nu)) : L(\lambda)] \in \BN_0$ for all $\nu$.  Note that $\CF_{c,\Bs}$ is invariant under $(k-s)$-Ind for all $\Bs \in \B_s$.  Thus applying $\pie_{\Bt}$ to the last equation above, we get by \autoref{genind}
$$
(k-s)\text{-Ind} ([(P(\lambda))]_{\Bt}) = [P(\mu)]_{\Bt} + \sum_{\substack{\nu\in\CP_{\Bt} \\ |\nu| = |\mu| \\ \nu < \mu}} L_{\lambda}^{\nu} [P(\nu)]_{\Bt}.
$$
Hence
\begin{align*}
 [P(\mu)]_{\Bt}- \E(G_{\Bt}(\mu)) & = (k-s)\text{-Ind}([P(\lambda)]_{\Bt}) - \E(f_k(G_{\Bt}(\lambda))) \\
&\qquad \qquad + \sum_{\substack{\nu\in\CP_{\Bt} \\ |\nu| = |\mu| \\ \nu < \mu}} \left(L_\lambda^\nu(1) \E(G_{\Bt}(\nu))- L_\lambda^\nu [P(\nu)]_{\Bt}\right) \\
 & =
\sum_{\substack{\nu\in\CP_{\Bt} \\ |\nu| = |\mu| \\ \nu < \mu}} \left(L_\lambda^{\nu}(1) -L_\lambda^\nu \right) \E(G_{\Bt}(\nu)), \tag{$*$}
\end{align*}
using the fact that $\E$ intertwines $f_k$ and $(k-s)$-Ind, and applying induction hypothesis to $\lambda$ and $\nu$.

By \autoref{Klthm} and \autoref{LBt}, we see that  $L_{\lambda}^\nu(1) \leq L_{\lambda}^{\nu}$ for all $\nu \in \CP_{\Bt}$.  If $L_{\lambda}^\nu(1) \ne L_{\lambda}^{\nu}$ for some $\nu \in \CP_{\Bt}$, let $\rho \in \CP_{\Bt}$ be maximal with respect to $\geq_{J_p}$ such that $L_{\lambda}^\rho(1) \ne L_{\lambda}^{\rho}$.  Then $L_{\lambda}^{\rho}(1) < L_\lambda^{\rho}$, and, for all $\nu \in \CP_{\Bt}$, $L_{\lambda}^{\nu}(1) - L_{\lambda}^{\nu} = 0$ if $\nu >_{J_p} \rho$ while $d_{\rho\nu}(q) = 0$ if $\nu \not\geq_{J_p} \rho$ by \autoref{T:appendix}(3).
Thus, comparing the coefficient of $\beta_s(\rho)$ in both sides of ($*$), we get by \autoref{T:decomp}
$$
d_{\rho\mu} - d_{\rho\mu}(1) = \sum_{\substack{\nu \in \CP_{\Bt} \\ |\nu| = |\mu| \\ \nu < \mu}} (L_\lambda^\nu(1) - L_\lambda^{\nu}) d_{\rho\nu}(1) = L_\lambda^{\rho}(1)-L_\lambda^{\rho}.
$$
But $d_{\rho\mu} \geq d_{\rho\mu}(1)$ by \autoref{decompnumbers}(3) while $L_\lambda^{\rho}(1) < L_\lambda^{\rho}$, giving us a contradiction.
% Since $d_{\rho\mu} \geq d_{\rho\mu}(1) =  d^{\,\Bt}_{\rho\mu}(1)$ for all $\rho\in\CP_{\Bt}$, by (adjustment matrix + VV + earlier theorem), we deduce that $L_\nu(1)-L_\nu=0$ for all $\nu$, and
Therefore, $L_\lambda^\nu(1) = L_\lambda^\nu$ for all $\nu \in \CP_{\Bt}$, and we get from ($*$) that $[P(\mu)]_{\Bt}=\E(G_{\Bt}(\mu))$, as desired.
\end{proof}

\begin{cor} \label{C:Schur}
Let $\Bn = (n_1, \dotsc, n_r)$ be a tuple of positive integers summing to $p$, with $n_j \in \{1,2\}$ for all $j$.
Then $d_{\lambda\mu} = d_{\lambda\mu}(1)$ for all $\lambda,\mu \in \CP_{\Bt}$  and $\Bt \in \BT_{\Bn}$.
\end{cor}

\begin{proof}
This follows immediately from \autoref{T:decomp} and \autoref{T:Schur}.
\end{proof}

We note that if $\Bn = (n_1, \dotsc, n_r)$ is a tuple of positive integers summing to $p$, with $n_j \in \{1,2\}$ for all $j$, then $d^{\,\Bt}_{\lambda\mu}(q)$, and hence $d_{\lambda\mu}(q)$ and $d_{\lambda\mu}$ by virtue of \autoref{T:Schur} and \autoref{C:Schur}, can be described by the closed formulas found by Frenkel and Khovanov using \autoref{T:product} for all $\Bt \in \BT_{\Bn}$, $\lambda, \mu \in \CP_{\Bt}$.

\begin{rem} \autoref{C:Schur} does not hold for arbitrary $\Bn$, even under James's hypothesis that the size of the partitions $\lambda$ and $\mu$ is strictly less that $p^2$ \cite{Jam}. Williamson \cite{W} produces counterexamples in which $\lambda,\mu\in \CP_\Bt$ are partitions of
	$p\binom{N}{2}<p^2$, and where $\Bn=(p)$ and $\Bt$ is given by
	$$\Bt(i)=
	\begin{cases}
	p & \text{if } i\leq 0 \\
	1 & \text{if } 1 \leq i \leq N \\
	0 & \text{if } i \geq N+1.
	\end{cases}$$
	\end{rem}

%In particular, we obtain Theorem 4.6 of \cite{TT}.

%\begin{cor}
%Let $\lambda,\mu \in \CP$ such that $\lambda$ is obtained from $\mu$ by moving some nodes of the same $p$-residue (i.e\ $\lambda$ is obtained from $\mu$ by removing some removable nodes and adding some addable nodes, all of which having the same $p$-residue).  Then $$d_{\lambda\mu} = d_{\lambda\mu}(1).$$
%\end{cor}

%In fact, more is true.

Two residue classes $a,b$ modulo $p$ are {\em adjacent} if and only if $a-b \equiv_p \pm 1$.

\begin{lemma} \label{12}
Let $\lambda, \mu \in \CP$, and let $I$ be the set of $p$-residues of the nodes in $([\lambda] \setminus [\mu]) \cup ([\mu] \setminus [\lambda])$.  Suppose that no two elements of $I$ are adjacent. Then there exist $s \in \{0,1\}$ and $\Bn = (n_1,\dotsc, n_r)$, with $n_j \in \{1,2\}$ for all $j$ and $\sum_{j=1}^r n_j = p$, such that $\t^{\Bn}_{\beta_s(\lambda)} = \t^{\Bn}_{\beta_s(\mu)}$ and $I = \{ \sum_{j=1}^i n_j -s -1 \mid n_i = 2 \}$.
\end{lemma}

\begin{proof}
Let $s = 1$ if $0 \in I$ and $s=0$ if $0 \notin I$.  Let $\Bn = (n_1,\dotsc,n_r)$ be a sequence of positive integers summing to $p$ such that $\{ \sum_{j=1}^i n_j - s \mid s \leq i \leq r-1+s \} = \{ 0,1,\dotsc, p-1 \} \setminus I$.  Since the elements of $I$ are pairwise non-adjacent, we see that $n_j \leq 2$ for all $j$.  Furthermore, $\t^{\Bn}_{\beta_s(\lambda)} = \t^{\Bn}_{\beta_s(\mu)}$ by \autoref{characterisation} and $I =\{\sum_{j=1}^i n_j -s -1 \mid n_i = 2\}$.
\end{proof}

If $\mu \in \CP$, and $\lambda$ is obtained from $\mu$ by removing a removable node of $p$-residue $i$ and adding an addable node of $p$-residue $i$, we say that $\lambda$ is obtained from $\mu$ by {\em moving} a node of $p$-residue $i$.

\begin{thm} \label{decomp*}
Let $\mu \in \CP$ and suppose that $\lambda$ is obtained from $\mu$ by moving some nodes whose $p$-residues are pairwise non-adjacent.  Let $I$ be the subset of residue classes modulo $p$ which occur as $p$-residues of the nodes moved to obtain $\lambda$.  For each $i \in I$, let $\lambda(i)$ be the partition obtained from $\mu$ by moving only those nodes with $p$-residue $i$.  Then
$$
d_{\lambda\mu} = \prod_{i\in I} d_{\lambda(i),\mu}.
$$
\end{thm}

\begin{proof}
By \autoref{12}, there exist $s \in \{0,1\}$ and $\Bn = (n_1,\dotsc,n_r)$, with $n_j \in \{1,2\}$ for all $j$ and $\sum_{j=1}^r n_j = p$, such that $\t^{\Bn}_{\beta_s(\lambda)} = \t^{\Bn}_{\beta_s(\mu)}$.  Furthermore, by \autoref{characterisation}, $\t^{\Bn}_{\beta_s(\lambda)} = \t^{\Bn}_{\beta_s(\lambda)(i)} = \t^{\Bn}_{\beta_s(\mu)}$ for all $i \in I$.  By \autoref{T:product} and adopting the notations there, we have
$$
d^{\,p}_{\lambda\mu}(q) = \prod_{j=1}^r d^{\,n_j}_{\lambda^{(j)}\mu^{(j)}}(q) = \prod_{j\,: \,n_j =2} d^{\,2}_{\lambda^{(j)}\mu^{(j)}}(q),
$$
since $\lambda^{(j)} = \mu^{(j)}$ if $n_j = 1$.  For each $j$ such that $n_j = 2$, let $i_j = \sum_{i=1}^j n_i -s -1$.  Then
$$
d^{\,p}_{\lambda(i_j),\mu}(q) = d^{\,2}_{\lambda^{(j)}\mu^{(j)}}(q) \left(\prod_{k \ne j} d^{\,n_k}_{\mu^{(k)}\mu^{(k)}}(q) \right) = d^{\,2}_{\lambda^{(j)}\mu^{(j)}}(q),
$$
by \autoref{T:product} and \autoref{T:appendix}(1).  Thus,
$$
d_{\lambda\mu} = d^{\,p}_{\lambda\mu}(1)= \prod_{j\,: \,n_j = 2} d^{\,2}_{\lambda^{(j)}\mu^{(j)}}(1) = \prod_{j\,: \, n_j = 2} d^{\,p}_{\lambda(i_j),\mu}(1) = \prod_{i \in I} d^{\,p}_{\lambda(i),\mu}(1) = \prod_{i \in I} d_{\lambda(i),\mu}$$
by \autoref{12} and \autoref{C:Schur}.
\end{proof}

\begin{thm} \label{branching}
Let $\mu \in \CP$ and let $\lambda$ be the partition obtained from $\mu$ by removing a removable node $\mathfrak{n}$ of $p$-residue $i$ and moving $m$ other nodes, such that the $p$-residues of these nodes (including $\mathfrak{n}$) are pairwise non-adjacent.  Then
$$
[i\text{-}{\rm Res}\, (L(\mu)) : L(\lambda)] = 0,
$$
unless $m=0$ and $\mathfrak{n}$ is a normal node of $\mu$.
\end{thm}

\begin{proof}
By \autoref{12}, there exist $s \in \{0,1\}$ and $\Bn = (n_1,\dotsc, n_r)$ with $n_j \in \{1,2\}$ for all $j$ and $\sum_{j=1}^r n_j = p$ such that $\t^{\Bn}_{\beta_s(\lambda)} = \t^{\Bn}_{\beta_s(\mu)}$.  Let $\Bt = \t^{\Bn}_{\beta_s(\lambda)}$.   Applying $\pie_{\Bt}$ to
$i\text{-Ind}([P(\lambda)]) = \bigoplus_{\nu \in \CP} L_\lambda^\nu [P(\nu)]$, where $L_\lambda^\nu = [i\text{-}{\rm Res}\, (L(\nu)) : L(\lambda)]$ for all $\nu \in \CP$, we get by \autoref{genind}(2)
$$
i\text{-Ind}([P(\lambda)]_{\Bt}) = \bigoplus_{\nu \in \CP_{\Bt}} L_\lambda^\nu [P(\nu)]_{\Bt}.$$
On the other hand, applying $\E$ to $f_{s+i}(G_{\Bt}(\lambda)) = \sum_{\nu \in \CP_{\Bt}} L_\lambda^{\nu}(q) G_{\Bt}(\nu)$, where $L_{\lambda}^{\nu}(q) \in A$ for all $\nu \in \CP_{\Bt}$, we get by \autoref{T:Schur}
$$
i\text{-Ind}([P(\lambda)]_{\Bt}) = \bigoplus_{\nu \in \CP_{\Bt}} L_\lambda^\nu(1) [P(\nu)]_{\Bt}.
$$
Thus, $L_{\lambda}^{\nu}(1) = L_\lambda^\nu$ for all $\nu \in \CP_{\Bt}$.  The statement now follows from \autoref{LBt}.
\end{proof}

\begin{rem} \hfill
\begin{enumerate}
\item Exactly the same proofs of \autoref{decomp*} and \autoref{branching} in fact prove the respective results hold in a slightly more general setting: whenever any two nodes in the symmetric difference of $[\lambda]$ and $[\mu]$ do not have adjacent $p$-residues.  But the additional information is actually trivial---they can easily be seen to hold with our current knowledge of decomposition numbers and branching coefficients of Schur algebras.
\item We have deliberately left out in \autoref{branching} the closed formula for $[i\text{-}{\rm Res}\, (L(\mu)) : L(\lambda)]$ when $\lambda$ is obtained from $\mu$ by removing a normal node of residue $i$ (which is one plus the number of normal nodes of $\mu$ of residue $i$ and to the right of the normal node removed to obtain $\lambda$, see \autoref{Klthm}) so as not to give a false impression that we have an alternative and independent proof of \autoref{Klthm}.
\end{enumerate}
\end{rem}

\end{document}